\documentclass[11pt]{amsart}
\usepackage{amsmath}
\usepackage{amssymb}
\usepackage{amsfonts}
\usepackage{mathrsfs}
\usepackage{graphicx}
\usepackage{epsfig}
\usepackage{enumitem}
\usepackage[T1]{fontenc}
\usepackage{enumitem}
\usepackage[a4paper,centering,top=25mm,bottom=25mm,bindingoffset=0mm]{geometry}
\numberwithin{equation}{section}       

 \theoremstyle{plain}
\newtheorem{theorem}{Theorem}[section]

\newtheorem{prop}{Proposition}[section]

\newtheorem{coro}[prop]{Corollary}
\newtheorem{lemma}[prop]{Lemma}
\newtheorem{fact}[prop]{Fact}

\newtheorem*{mainlemma}{Main Lemma}

\usepackage{hyperref}

\theoremstyle{definition}
\newtheorem{definition}[prop]{Definition}

\theoremstyle{remark}
\newtheorem{remark}[prop]{Remark}

\newtheoremstyle{citing}
  {3pt}
  {3pt}
  {\itshape}
  {}
  {\bfseries}
  {.}
  {.5em}
  {\thmnote{#3}}

\theoremstyle{citing}

\DeclareMathAlphabet{\mathpzc}{OT1}{pzc}{m}{it} 

\newcommand{\C}{\mathbb{C}}
\newcommand{\D}{\mathbb{D}}

\newcommand{\R}{\mathbb{R}}

\newcommand{\Z}{\mathbb{Z}}

\newcommand{\teta}{\widetilde{\teta}}

\newcommand{\eps}{\varepsilon}

\DeclareMathOperator{\diam}{diam}

\newcommand{\CC}{\overline{\C}}

\newcommand{\Crit}{\mathcal{C}}

\def\HH{\mathcal{H}}
\begin{document}
\title[]{Monotonicity of entropy and positively oriented transversality for families of interval maps}
\author{Genadi Levin, Weixiao Shen and Sebastian van Strien}

\date{\today}

\begin{abstract}
In this paper we will develop a very general approach which shows that critical relations
of holomorphic maps on the complex plane unfold transversally in a {\lq\lq}positively oriented{\rq\rq} way. We will mainly illustrate this approach to obtain
transversality for a wide class of one-parameter families of interval maps,
for example maps with flat critical points,  piecewise linear maps, maps with discontinuities but also for families of maps
with  complex analytic extensions such as certain polynomial-like maps.
\end{abstract}

\maketitle
\tableofcontents
\section{Introduction}

In this paper we will develop a new method for showing transversality properties of families $f_t$ of maps of the complex
plane so that $f_0$ has a finite invariant marked set. Surprisingly,  this method works even when dealing with holomorphic maps
with domain and range open subsets of the complex plane.
For the unicritical family $z\mapsto z^{2d}+c$, this method gives a new and simple proof of well-known results, see Section~\ref{sec:oneparameter}.
The method also apply to many other families.

\subsection{Transversality through holomorphic motions}
Before stating specific theorems which follow from the approach developed in this paper, let
us discuss the general philosophy.  The setting in this paper is to consider rather general maps on open subsets of the complex plane with a finite
forward invariant marked set, for example the postcritical set.  These maps do not necessarily have to be rational or transcendental.
The aim is to show that critical relations unfold transversally,
by considering a holomorphic motion along the marked set.
By lifting this holomorphic motion by the dynamics we obtain a corresponding transfer operator $\mathcal A$. 
 It turns out that one has the following implications:
 \begin{equation*}
\mbox{lifting property} \implies
 \mbox{spec}(\mathcal A)\subset \mathbb D   \implies \mbox{transversality properties}.\end{equation*}
In Part A  we will define the {\lq}lifting property{\rq} and prove
these implications.  In fact, we will obtain {\lq}positively oriented transversality{\rq} in a sense which is made
precise in  Subsection~\ref{subsec:intro-on-appendix}.

In Part B we will show this lifting property follows from a separation property.
In this way, we derive in Section~\ref{sec:oneparameter}
 transversality for many families of interval maps but also for a wide class of one-parameter
families of the form $f_\lambda(x)=f(x)+\lambda$ and $f_\lambda(x)=\lambda f(x)$.
For example, as an easy application, we will recover known transversality results for the family of quadratic maps,
and in Subsection~\ref{subsec:additive} we partially address some conjectures from the 1980's about such families of maps.
In Section~\ref{sec:separation} we present this set-up in a rather general framework,
defining in this setting a {\em separation property} and prove that this property implies the lifting property.
In Section \ref{sec:sinearnold} we will use this separation property to obtain
transversality for  maps of the circle (maps from the generalised Arnol'd family).

In Part C, we will study the family $x\mapsto |x|^\ell+c$. When $\ell$ is not an even integer, we have not been able to prove the lifting property in general.
Nevertheless we will obtain the lifting property under additional assumptions in   Sections \ref{sec:finiteorder}-\ref{sec:finiteoddorder}.

In Part  D,
we show that the methods developed in this paper also apply to other families, which do not necessarily have the separation property.
For example, in Section~\ref{sec:rationalmaps} we consider the setting of  general
polynomial families and rational families,
obtain the lifting property using the Measurable Riemann Mapping theorem,
 and  thus transversality which is {\lq}positively oriented{\rq}.
In Sections~\ref{subset:linear} and \ref{sec:familiesdiscont}  we apply the methods from this paper
to obtain positively oriented transversality for piecewise linear interval maps and interval maps with discontinuities (i.e. Lorenz maps).

\subsection{Results for one-parameter families of (possibly non-analytic) unimodal interval maps of the form $f_c(x)=f(x)+c$}
Let $\mathcal{U}$ denote the collection of unimodal maps $f: \R\to \R$ which are strictly decreasing in $(-\infty,0]$ and strictly increasing in $[0,\infty)$. Given $f\in\mathcal{U}$, we are interested in the bifurcation in the family $f_c(x)=f(x)+c$, $c\in\R$ and in particular the problem whether the Milnor-Thurston kneading sequences depend on $c$ monotonically.

The case $f_c(x)=x^2+c$ was solved in 1980s as a major result in unimodal dynamics. By now there are several proofs, see~\cite{MT, Su, D, Tsu0,Tsu1}. All these proofs use complex analytic methods and rely on the fact that $f_c$ extends to a holomorphic map on the complex plane. These methods work well for $f_c(x)=|x|^\ell+c$ when $\ell$ is a positive even integer but break down for general $\ell$ and other families of non-analytic unimodal maps.  No approach using purely real-analytic method has so far been successful in proving this monotonicity theorem fully.

The complex analytic method developed in Section~\ref{sec:lifting}
(which in the unimodal reduces to a few pages), shows that it is sufficient to check a certain separation property to obtain
monotonicity  for families of the form $f_c(x)=f(x)+c$,  see  Theorem~\ref{single}.
In Subsection~\ref{subsec:additive} we obtain from  this a new elementary proof
of the well-known monotonicity theorem for $f_c(x)=|x|^\ell+c$ when $\ell$ is an even positive integer, but more importantly also
monotonicity 
for  families of some non-analytic unimodal maps:

%


\begin{theorem} \label{thm:flat}
Fix real numbers $\ell\ge 1$ and $b> 2(e\ell)^{1/\ell}$ and consider the family
$$f_c(x)= b e^{-1/|x|^\ell} +c, \,\, c\in\R$$
of unimodal maps. Let $\beta\in (0, \ell^{1/\ell})$ be the solution of the equation 
$$f_{-\beta}(\beta)=\beta.$$ Then 
the kneading sequence $\mathcal{K}(f_c)$ is monotone increasing in $c\in [-\beta, \infty)$  and the positive transversality condition
(\ref{eq:trans}) below.
\end{theorem}

Recall that the {\em Milnor-Thurston kneading sequence} of $f\in\mathcal{U}$ is defined as a word  $\mathcal{K}(f)=i_1 i_2\cdots\in \{1, 0, -1\}^{\Z^+}$, where
$$i_k=\left\{\begin{array}{ll}
1 & \mbox{ if } f^k(0)>0\\
0 & \mbox{ if } f^k(0)=0\\
-1 & \mbox{ if } f^k(0)<0.
\end{array}
\right.
$$
For $g\in\mathcal{U}$ with $\mathcal{K}(g)=j_1j_2\cdots$, we say that $\mathcal{K}(f)\prec \mathcal{K}(g)$ if there is some $n\ge 1$ such that
$i_k=j_k$ for all $1\le k<n$ and $\prod_{k=1}^n i_k < \prod_{k=1}^n j_k$.
Thus given $f\in\mathcal{U}$, to prove monotonicity of kneading sequence in a family $f_c(x)=f(x)+c$, $c\in\R$, it suffices to show that one of the following properties holds:
\begin{itemize}
\item ({\bf Rigidity}) if $f_{c}$ has $0$ as a periodic point and $f_{\hat{c}}$ has the same kneading sequence as $f_{c}$, then  $c=\hat{c}$;
\item ({\bf {\lq\lq}Positive{\rq\rq} transversality}) if $f_{c_*}$ has $0$ as a periodic point of period $q$, then
\begin{equation}\frac{\frac{d}{d c} f_c^q(0)\left.\right|_{c=c_*}}{Df_{c_*}^{q-1}(c_*)}= \sum_{n=0}^{q-1}\frac{1}{Df_{c_*}^i(c_*)}>0.\label{eq:trans}\end{equation}
\end{itemize}
\begin{remark}
Equation (\ref{eq:trans}) implies that if $0$ has (precisely) period $q$ at some parameter $c_*$, then
$$\begin{array}{rl} \frac{d}{dc}f^q_c(0)\bigr \vert_{c=c_*}<0 & \mbox { if } f^q_{c_*} \mbox{ has a local maximum at }0,\\
\frac{d}{dc} f^q_c(0)\bigr \vert_{c=c_*}>0 & \mbox { if }  f^q_{c_*} \mbox{ has a local minimum at }0.\end{array}
$$
Hence  the multiplier $\lambda(c)$ of the (local) analytic continuation $p(c)$ of this periodic point of period $q$ is strictly increasing.
Note that there is a result of Douady-Hubbard which asserts that in each hyperbolic component of the family of quadratic maps, the
multiplier of the periodic attractor is a univalent function of the parameter.
Proving (\ref{eq:trans}) complements this by
also showing that on the real line the multiplier of the periodic point is increasing.
The approach to prove monotonicity via the inequality (\ref{eq:trans})  was also previously  used
by Tsujii \cite{Tsu0,Tsu1} for real maps of the form $z\mapsto z^2+c$, $c\in \R$.
\end{remark}

When $f_c(x)=|x|^\ell+c$,  and $\ell$ is not an integer, we have not been able able to prove the lifting property.
The next theorem gives monotonicity 
when $\ell$ is a large real number (not necessarily an integer),
but only
if not too many points in the critical orbit are in the orientation reversing branch.

\begin{theorem}\label{thm:finiteorder}
Let  $\ell_-,\ell_+\ge 1$ and consider the family of unimodal maps $f_c=f_{c,\ell_-,\ell_+}$ where
$$f_c(x)=\left\{\begin{array}{ll}
|x|^{\ell_-}+c & \mbox{ if } x\le 0\\
|x|^{\ell_+}+c & \mbox{ if } x\ge 0.
\end{array}
\right.
$$
For any integer $L\ge 1$ there exists $\ell_0>1$
so that for any $q\ge 1$ and any
periodic kneading sequence   $\bold i=i_1i_2\cdots\in \{-1,0,1\}^{\Z^+}$ of period $q$
so that
\begin{equation}\#\{1\le j< q ; i_j =-1 \}\le L,\label{assum:comb}\end{equation}
and any pair $\ell_-,\ell_+\ge \ell_0$ there is at most one $c\in\R$ for which the kneading sequence
of $f_c$ is equal to $\bold i$. Moreover,
\begin{equation}\label{fiortrans}
\sum_{n=0}^{q-1} \frac{1}{Df_c^n(c)}>0.
\end{equation}
\end{theorem}

The proof of this theorem uses delicate geometric arguments,  see Section~\ref{sec:finiteorder}.
In Section~\ref{sec:finiteoddorder} an analogue of this theorem is proved for the case that $\ell$ is an arbitrary odd integer,
but under a stronger assumption on the combinatorics of the critical orbit.

\subsection{Results for other families of interval maps}
In Subsection~\ref{subsec:multiplicative} we introduce a rather large class $\mathcal{E}$  of interval maps
with only critical values at $1$ and possibly at $0$ and with a minimal  $c>0$ so that $f$ has a positive local maximum at $c$.

\begin{theorem} Assume that $f\in \mathcal{E}$ and that $c$ is a periodic point for a map
of the form $f_\lambda(x)=\lambda f(x)$.
Then the transversality property (\ref{eq:trans}) holds.
\end{theorem}

The analogous result also holds for a related class $\mathcal{E}_o$.
Examples of maps in $\mathcal{E}$ and  $\mathcal{E}_o$  are given in Subsection~\ref{subsec:multiplicative}.

In Section~\ref{sec:sinearnold} we prove transversality for the  Arnol'd family $x\mapsto x + a +b \sin(2\pi x)$.

In this paper  we shall also consider the case when there are several critical points,
all of which eventually periodic. In this case
 the transversality condition (\ref{eq:trans})
has to be replaced by a more general transversality condition (\ref{eq:trans2}).

Our approach also applies to the setting of polynomials and rational maps, see Section \ref{sec:rationalmaps},
families of  piecewise linear maps
and to families of intervals maps with discontinuities (i.e. Lorenz maps), see Sections~\ref{subset:linear} and \ref{sec:familiesdiscont}.

\medskip
Since the polynomial and rational case is so important, in a separate paper \cite{LSvS} we have given a very elementary proof
of this and another theorem, but without the sign in (\ref{fiortrans}) and the corresponding case
when there are several critical point, see (\ref{eq:trans2}).  In that
paper we  also allow the postcritical set to be {\em infinite}. For an alternative discussion on transversality in the context of rational maps
with a finite postcritical set, see
\cite{Ep} and also \cite{BE}.

\subsection{Positively oriented transversality}
\label{subsec:intro-on-appendix}

%

An important feature of our prove is the sign in the above inequality. In the case when there are several critical points,
we obtain a corresponding inequality for some matrix.  That the sign of the determinant is positive,  means that the intersection of the algebraic sets
$R_j(g)=0$, $j=1,\dots,\nu$ corresponding to each of the critical relations is not only transversal, but that the intersection pattern is {\em everywhere  {\lq}positively oriented{\rq}}.

It would be interesting to know whether the sign in (\ref{eq:trans2}) makes it possible to simplify the proof in \cite{BS} of Milnor's conjecture.
This conjecture is about  the space of real polynomials with only real critical points, all of which non-degenerate,
and asks whether the  level sets of constant topological entropy are connected. The proof of this conjecture in \cite{BS}
relies on quasi-symmetric rigidity, but does having a positive sign in (\ref{eq:trans2}) everywhere allow for a simplification
of the proof of this conjecture?

\medskip

{\bf Acknowledgment.} We are indebted to Alex Eremenko for very helpful discussions concerning Subsection~\ref{subsec:multiplicative}.
This project was partly supported by ERC AdG grant no: 339523 RGDD.

\part*{Part A: A new method}
\section{Transversality and lifting holomorphic families}\label{sec:lifting}

In this section we study a transfer operator $\mathcal A$ associated to the analytic deformation
of a {\lq}marked map{\rq}, and show that if  $1$ is not an eigenvalue of $\mathcal A$
then a certain transversality condition holds (related to, in applications later on in this paper, to critical relations), see Section~\ref{subsec:transA}.
If the spectrum of $\mathcal A$ is inside the unit circle,  we will obtain additional information
about transversality, see Section~\ref{subsec:trans-sign}.   It turns out that if a certain lifting property of
holomorphic families holds, then the spectrum of $\mathcal A$ is inside the unit circle, see Section~\ref{subsec:lifting}.
If $1$ is an eigenvalue of
$\mathcal A$ then the set where one critical relation holds forms an analytic variety, see Theorem~\ref{thm:1eigen}
in Subsection~\ref{subsec:Aeig1}.

\subsection{Transversality of a marked map with respect to a holomorphic deformation}

\def\MM{\mathcal M}
A {\em marked map} is a map $g$ from the union of a finite set $P_0$ and an open set $U$ in $\C$ into $\C$ such that
\begin{itemize}
\item there exists a finite set $P\supset P_0$ such that $g(P)\subset P$ and $P\setminus P_0\subset U$;
\item $g|U$ is holomorphic and $g'(x)\not=0$ for $x\in P\setminus P_0$.
\end{itemize}
Let $c_{0,j}, j=1,2,\ldots, \nu$ denote the distinct points in $P_0$ and write
${\bf c}_0={\bf c}_0(g)=(c_{0,1},\dots ,c_{0,\nu})$
and ${\bf c}_1={\bf c}_1(g)=(g(c_{0,1}),\dots, g(c_{0,\nu})):=(c_{1,1},c_{1,2},\ldots, c_{1,\nu})$.

A {\em local holomorphic deformation} of $g$
is a triple $(g,G, \textbf{p})_W$ with the following properties:
\begin{enumerate}
\item $W$ is an open connected subset of $\C^\nu$ containing  $\textbf{c}_1(g)$;
\item $\textbf{p}=(p_1,p_2,\ldots, p_\nu):W\to \C^\nu$ is a holomorphic map,
so that $\textbf{p}(\textbf{c}_1)={\bf c}_0(g)$ (and so all coordinates of $\textbf{p}(\textbf{c}_1)$ are distinct).
\item $G: (w,z)\in W\times U \mapsto (w,G_w(z))\in W\times \C$ is a holomorphic map such that $G_{\textbf{c}_1}=g$.
\end{enumerate}
Let us fix $(g,G, \textbf{p})_W$ as above.
Since $g(P)\subset P$ and $P$ is a finite set,
for each $j=1,2,\ldots, \nu$, one of the following holds:
\begin{itemize}
\item There exists a positive integer $q_j$ and $\mu(j)\in \{1,2,\ldots,\nu\}$ such that
$g^{q_j}(c_{0,j})=c_{0,\mu(j)}$ and $g^k(c_{0,j})\not\in P_0$ for each $1\le k<q_j$;
\item There exist
positive integers $l_j<q_j$ such that
$g^{q_j}(c_{0,j})=g^{l_j}(c_{0,j})$ and $g^k(c_{0,j})\not\in P_0$ for all $1\le k\le q_j$.
We assume in the following that $l_j$ and $q_j$ are minimal with this property.
\end{itemize}
Relabelling these points $c_{0,j}$, we assume that there is $r$ such that
the first alternative happens for all $1\le j\le r$ and the second alternative happens for $r<j\le \nu$.

Define a map
$$\mathcal{R}=(R_1, R_2, \dots, R_\nu)$$
from a neighbourhood of $\textbf{c}_1\in \C^\nu$ into $\C^\nu$ as follows:
for $1\le j\le r$, $$R_j(\textbf{w})=G_{\textbf{w}}^{q_j-1}(w_j)-p_{\mu(j)}(\textbf{w})$$ and for $r<j\le \nu$, $$R_j(\textbf{w})=G_{\textbf{w}}^{q_j-1}(w_j)-G_{\textbf{w}}^{l_j-1}(w_j),$$
where $\textbf{w}=(w_j)_{j=1}^\nu$.

\begin{definition}\label{def:unfoldtrans}
We say that the holomorphic deformation $(g,G, \textbf{p})_W$ of $g$ {\em unfolds tranversally}, if the Jacobian matrix
$D\mathcal{R}(\textbf{c}_1)$ is invertible.
\end{definition}

A marked map $g$ is called {\em real} if $P\subset \R$ and for any $z\in U$ we have $\overline{z}\in U$ and $\overline{g(z)}=g(\overline{z})$.  Similarly, a local holomorphic deformation $(g, G, p)_W$ of a real marked map $g$ is called {\em real} if
for any $w=(w_1, w_2, \ldots, w_\nu)\in W$, $z\in U$ and $j=1,2,\ldots,\nu$, we have
$\overline{w}= (\overline{w_1}, \overline{w_2},\ldots, \overline{w_\nu})\in W$,
$$G_{\overline{w}}(\overline{z})=\overline{G_w(z)}, \text{ and } p_j(\overline{w})=\overline{p_j(w)}.$$

\begin{definition}
Let $(g, G, \textbf{p})_W$ be a real local holomorphic deformation of a real marked map $g$.
If
\begin{equation}\frac{\det (D\mathcal{R}(\textbf{c}_1))}{\prod_{j=1}^\nu Dg^{q_j-1}(c_{1,j})}>0.\label{eq:trans2}
\end{equation}
holds we say that  the unfolding $(g,G, \textbf{p})_W$ satisfies {\em the `positively oriented' transversality property}.
\end{definition}
Note that inequality (\ref{eq:trans2}) implies that the family unfolds transversally as in Definition~\ref{def:unfoldtrans}.
Inequality (\ref{eq:trans2}) is the generalisation of (\ref{eq:trans}) in the setting of several critical points,
all of which eventually periodic and if the critical points are allowed to move with the parameters.
In the case that the map has only one critical point which is periodic and which does not depend on the parameter,
(\ref{eq:trans2}) reduces to (\ref{eq:trans}).

\subsection{A transfer operator associated to a deformation of a marked map}\label{subsec:21}
Let  $\Lambda$ be a domain in $\C$ and $\ast\in \Lambda$.

A {\em holomorphic motion} of $g(P)$ {\em over} $(\Lambda,\ast)$ is a family of injections $h_{\lambda}: g(P)\to \C$, $\lambda\in\Lambda$,  
such that $h_\ast=id_{g(P)}$ and $\lambda\mapsto h_{\lambda}(x)$ is holomorphic for each $x\in g(P)$. Given an open neighbourhood $\Lambda_0$ of $\ast$ in $\Lambda$ and a holomorphic motion $\widehat{h}_\lambda(x)$ of $g(P)$ over $(\Lambda_0, \ast)$, 
we say that $\hat{h}_\lambda$ is a {\em lift}  of $h_{\lambda}$ {\em over} $\Lambda_0$ with respect to  $(g,G, \textbf{p})_W$ if the following holds when $d(\lambda,\ast)$ is small enough:
\begin{itemize}
\item For each $j=1,2, \dots, \nu$, with $c_{0,j}\in g(P)$, \
\begin{equation}\hat{h}_\lambda(c_{0,j})=p_j({\textbf{c}_1}(\lambda)),\label{eq:lift1}\end{equation}
    where $\textbf{c}_1(\lambda)=(h_{\lambda}(c_{1,1}), h_{\lambda}(c_{1,2}),\dots, h_{\lambda}(c_{1,\nu}))$;
\item for each $x\in g(P)\setminus P_0$, we have
\begin{equation}G_{\textbf{c}_1(\lambda)}(\hat{h}_{\lambda}(x))=h_{\lambda}(g(x)).\label{eq:lift2} \end{equation}
\end{itemize}
Here we use that $\textbf{c}_1(\lambda)\in W$ when $d(\lambda,\ast)$ is sufficiently small.
Clearly, locally any holomorphic motion $h_\lambda(x)$ of $P$ over $(\Lambda, \ast)$ has a lift under $(g, G, p)_W$, i.e. there is a holomorphic motion $\widehat{h}_\lambda$ over $(\Lambda_0, \ast)$, where $\Lambda_0$ is an open neighbourhood of $\ast$ in $\Lambda$ such that $\widehat{h}_\lambda$ is the lift of $h_\lambda$ over $\Lambda_0$.

Obviously there is a linear map $\mathcal{A}: \C^{\#g(P)}\to \C^{\#g(P)}$ such that whenever $\hat{h}_\lambda$ is a lift of $h_\lambda$, we have
$$\mathcal{A}\left(\left\{\frac{d}{d\lambda}h_\lambda(x)\left|\right._{\lambda=\ast}\right\}_{x\in g(P)}\right)
=\left\{\frac{d}{d\lambda}\hat{h}_\lambda(x)\left|\right._{\lambda=\ast}\right\}_{x\in g(P)}.$$
We will call $\mathcal{A}$ the {\em transfer operator} associated to the holomorphic deformation $(g,G, \textbf{p})_W$ of $g$.

If both $g$ and $(g, G, \textbf{p})_W$ are real, then $\mathcal{A}(\R^\nu)\subset \R^\nu$. In this case, we shall often consider real holomorphic motions, i.e. $\Lambda$ is symmetric with respect to $\R$, $\ast\in \R$ and $h_\lambda(x)\in \R$ for each $x\in g(P)$ and $\lambda\in \Lambda\cap\R$. Clearly, a lift of a real holomorphic motion is again real.

\subsection{Relating the transfer operator with transversality}\label{subsec:transA}

It turns out that transversality is closely related to the eigenvalues of $\mathcal{A}$:

\begin{lemma}\label{lem:21}  Assume the following holds: for any $r<j<j'\le\nu$ with $g^{q_j}(c_{0,j})= g^{q_{j'}}(c_{0,j'})$ we
 have $Dg^{q_j-l_j}(c_{l_j,j})\not=1.$
Then the following statements are equivalent:
\begin{enumerate}
\item $1$ is an eigenvalue of $\mathcal{A}$;
\item $D\mathcal{R}(\textbf{c}_1)$ is degenerate.
\end{enumerate}
\end{lemma}
\begin{proof} We first show that (1) implies (2), even without the assumption. So suppose that $1$ is an eigenvalue of $\mathcal{A}$ and let $\textbf{v}=(v(x))_{x\in g(P)}$ be an eigenvector associated with $1$.
For $t\in \D$, define $h_t(x)=x+t v(x)$ for each $x\in g(P)$ and $\textbf{w}(t)=(c_{1,j}+ t v(c_{1,j}))_{j=1}^\nu$.
Then for each $x\in g(P)\setminus P_0$,
\begin{equation}\label{eqn:21.1}
G_{\textbf{w}(t)}(h_t(x))-h_t(g(x))=O(t^2)
\end{equation}, and for each $x=c_{0,j}\in g(P)\cap P_0$, we have
\begin{equation}\label{eqn:21.2}
h_t(x)-p_j(\textbf{w}(t))=O(t^2).
\end{equation}
For each $1\le j\le \nu$, and each $1\le k<q_j$,
applying (\ref{eqn:21.1}) repeatedly, we obtain
\begin{equation}\label{eqn:21.3}
G_{\textbf{w}(t)}^k(h_t(c_{1,j}))=h_t(g^k(c_{1,j}))+O(t^2).
\end{equation}
Together with (\ref{eqn:21.2}), this implies that $$R_j(\textbf{w}(t))=O(t^2),$$
holds for all $1\le j\le \nu$. It remains to show $\textbf{w}'(0)\not=\textbf{0}$. Indeed, otherwise, by (\ref{eqn:21.3}), it would follow that $v(g^k(c_{1,j}))=(g^k)'(c_{1,j}) v(c_{1,j})=0$ for each $1\le j\le \nu$ and $1\le k<q_j$, and hence $v(x)=0$ for all $x\in g(P)$, which is absurd. We completed the proof that (1) implies (2).

Now let us prove that (2) implies (1) under the assumption of the lemma. Suppose that $D\textbf{R}(\textbf{c}_1)$ is degenerate. Then there exists a non-zero vector $(w_1^0, w_2^0,\cdots, w_\nu^0)$ in $\C^\nu$ such that $R_j(\textbf{w}(t))=O(t^2)$ as $t\to 0$,  where $\textbf{w}(t)=(w_j(t))_{j=1}^\nu=(c_{1,j}+ tw_j^0)_{j=1}^\nu$. We claim that $w_j^0=w_{j'}^0$ holds whenever $c_{1,j}=c_{1,j'}$, $1\le j,j'\le \nu$.
Indeed,

{\em Case 1.} If $1\le j\le r$ then $1\le j'\le r$ and $\mu(j)=\mu(j')$, $q_j=q_{j'}$. Then
$$G_{\textbf{w}(t)}^{q_j-1}(w_j(t))-G_{\textbf{w}(t)}^{q_j-1}(w_{j'}(t))=R_j(\textbf{w}(t))-R_{j'}(\textbf{w}(t))=O(t^2)$$
which implies that $w_j(t)-w_{j'}(t)=O(t^2)$, i.e. $w_j^0=w_{j'}^0$.

{\em Case 2.} If $r<j\le \nu$ then $r<j'\le \nu$ and $l_j=l_{j'}$, $q_j=q_{j'}$. Thus
$$G_{\textbf{w}(t)}^{q_j-1}(w_j(t))-G_{\textbf{w}(t)}^{q_j-1}(w_{j'}(t))=
G_{\textbf{w}(t)}^{l_j-1}(w_j(t))-G_{\textbf{w}(t)}^{l_j-1}(w_{j'}(t))+O(t^2),$$
which implies that
$$(Dg^{q_j-1}(c_{l_j,j})-Dg^{l_j-1}(c_{l_j,j}))(w_j(t)-w_{j'}(t))=O(t^2).$$
If such $j$ and $j'$  exist then $c_{l_j,j}$ is a hyperbolic periodic point,
hence $Dg^{q_j-1}(c_{l_j,j})\not=Dg^{l_j-1}(c_{l_j,j})$.
It follows that $w_{j}^0=w_{j'}^0$.

The claim is proved.
To obtain an eigenvector for $\mathcal{A}$ with eigenvalue $1$, define $v(c_{1,j})=w_j^0$,
$v(c_{0,j})=\frac{d}{dt} p_j(\textbf{w}(t))|_{t=0}$. For points $x\in g(P)\setminus P_0$, there is $j$ and $1\le s<q_j$ such that $x=g^s(c_{0,j})$, define $v(x)=\frac{d}{dt} G_{\textbf{w}(t)}^{s-1}(w_j(t))|_{t=0}$. Note that $v(x)$ does not depend on the choice of $j$ and $s$. (This can be proved similarly as the claim.)
\end{proof}

\subsection{The spectrum of $\mathcal A$ gives additional information concerning transversality}\label{subsec:trans-sign}
Define $D(\rho)=(D_{j,k}(\rho))_{1\le j,k\le \nu}$ as follows: Put
$$L_{k}(z)=\frac{\partial G_{\textbf{w}}(z)}{\partial w_k}\left.\right|_{\textbf{w}=\textbf{c}_1};\quad
p_{j,k}=\frac{\partial p_j}{\partial w_k}(\textbf{c}_1);$$
$$\mathcal{L}^0_{j,k}=0 \text{ and } \mathcal{L}^m_{j,k}=\sum_{n=1}^m\frac{\rho^n L_k(c_{n,j})}{Dg^n(c_{1,j})} \text{ for } m>0;$$
$$D_{jk}(\rho)=\delta_{jk} +\mathcal{L}^{q_j-1}_{j,k}-\rho^{q_j} \frac{p_{\mu(j),k}}{Dg^{q_j-1}(c_{1,j})}$$
when $1\le j\le r$ and
$$D_{jk}(\rho)=\delta_{jk}+\mathcal{L}^{q_j-1}_{j,k}
-\frac{\rho^{q_j-l_j}}{Dg^{q_j-l_j}(c_{l_j,j})}\left(\mathcal{L}^{l_j-1}_{jk}+\delta_{j,k}\right)$$
when $r<j\le \nu$. Note that
$$\det(D\mathcal{R}(\textbf{c}_1))=\prod_{j=1}^\nu Dg^{q_j-1}(c_{1,j})\det(D(1)).$$


We say that $\rho\in \C$ is an {\em exceptional value} if there exist $r<j<j'\le \nu$ such that
$\rho^{q_j-l_j}=Dg^{q_j-l_j}(c_{l_j,j})$ and $g^{q_j}(c_{0,j})=g^{q_{j'}}(c_{0,j'})$.
Note that for such $\rho$, $j$ and $j'$, $D_{jk}(\rho)=D_{j'k}(\rho)$ for all $k$ so that $\det(D(\rho))=0$. (But it may happen that $\det(I-\rho \mathcal{A})\not=0$.)
\begin{prop}\label{prop:non-exceptional}
For each non-exceptional $\rho\in \C$, we have
\begin{equation}
\det(I-\rho\mathcal{A})=0\Leftrightarrow \det (D(\rho))=0.\label{eq:non-exceptional}\end{equation}

\end{prop}
\begin{proof} For $\rho=0$, $\det(I)=\det(D(0))=1$. Assume $\rho\not=0$.
Define a new triple $(g^\rho, G^\rho, \textbf{p}^\rho)$ as follows.
\begin{itemize}
\item For each $x\in P\setminus P_0$, $G^\rho_{\textbf{w}}(z)=G_{\textbf{w}}(x)+\frac{Dg(x)}{\rho}(z-x)$ in a neighbourhood of $x$;
\item $g^\rho(x)=g(x)$ for each $x\in P_0$ and $g^\rho(z)= G_{\textbf{c}_1}^\rho(z)$ in a neighbourhood of $P\setminus P_0$;
\item $\textbf{p}^\rho(w)= \textbf{c}_0+ \rho \frac{\partial \textbf{p}}{\partial \textbf{w}}(\textbf{c}_1) \cdot (\textbf{w}-\textbf{c}_1).$
\end{itemize}
Let $\mathcal{A}^\rho$ be the transfer operator associated with the triple $(g^\rho, G^\rho, \textbf{p}^\rho)$. Then it is straightforward to check that
$$\mathcal{A}^\rho=\rho \mathcal{A}.$$
We can define a map $\mathcal{R}^\rho=(R^\rho_1, R^\rho_2,\cdots, R_\nu^\rho)$ for each $\rho\not=0$ in the obvious way:
$$R^\rho_j(\textbf{w})=(G_{\textbf{w}}^\rho)^{q_j-1}(w_j)-p^\rho_{\mu(j)}(\textbf{w})$$
for $1\le j\le r$ and
$$R^\rho_j(\textbf{w})=(G_{\textbf{w}}^\rho)^{q_j-1}(w_j)-(G_{\textbf{w}}^\rho)^{l_j-1}(w_j)$$
for $r<j\le \nu$.
As long as $\rho$ is non-exceptional for the triple $(g,G, \textbf{p})$, the new triple $(g^\rho, G^\rho, \textbf{p}^\rho)$ satisfies the assumption of Lemma~\ref{lem:21}, thus
$$\det(I-\rho\mathcal{A}^\rho)=0\Leftrightarrow D\mathcal{R}^\rho(\textbf{c}_1) \text{ is degenerate}.$$
Direct computation shows that the $(j,k)$-th entry of
$D\mathcal{R}^\rho (\textbf{c}_1)$ is
$D_{j,k}(\rho)Dg^{q_j-1}(c_{1,j})/\rho^{q_j-1}.$ Indeed,
for each $1\le j\le r$,
\begin{align*}
D^\rho_{j,k}(\rho)&=\frac{\partial (G^\rho_{\textbf{w}})^{q_j-1}(c_{1,j})}{\partial w_k}\left.\right|_{\textbf{w}=\textbf{c}_1}+\frac{Dg^{q_j-1}(c_{1,j})}{\rho^{q_j-1}}\delta_{jk}- \rho \frac{\partial p_{\mu(j)}}{\partial w_k}\\
&=\frac{Dg^{q_j-1}(c_{1,j})}{\rho^{q_j-1}}\left(\delta_{jk}+\sum_{n=1}^{q_j-1} \frac{\rho^n L_k(c_{n,j})}{Dg^n(c_{1,j})}-\rho^{q_j}\frac{p_{\mu(j),k}}{Dg^{q_j-1}(c_{1,j})}\right),
\end{align*}
and for $r<j\le \nu$,
\begin{align*}
& D^\rho_{jk}(\rho)\\
=&
\frac{\partial ((G_{\textbf{w}}^\rho)^{q_j-1}(c_{1,j})-(G_{\textbf{w}}^\rho)^{l_j-1}(c_{1,j}))}{\partial w_k}\left.\right|_{\textbf{w}=\textbf{c}_1}+\delta_{jk}
\left(\frac{Dg^{q_j-1}(c_{1,j})}{\rho^{q_j-1}}-\frac{Dg^{l_j-1}(c_{1,j})}{\rho^{l_j-1}}\right)\\
=&\frac{Dg^{q_j-1}(c_{1,j})}{\rho^{q_j-1}}\mathcal{L}^{q_j-1}_{j,k}
-\frac{Dg^{l_j-1}(c_{1,j})}{\rho^{l_j-1}}\mathcal{L}^{l_j-1}_{j,k}+\delta_{jk}
\left(\frac{Dg^{q_j-1}(c_{1,j})}{\rho^{q_j-1}}-\frac{Dg^{l_j-1}(c_{1,j})}{\rho^{l_j-1}}\right)
\end{align*}
Therefore $\det(I-\rho \mathcal{A})=0$ if and only if $\det(D(\rho))=0$.
\end{proof}

To illustrate the power of the previous proposition we state:

\begin{coro}[The transversality condition]
\label{real} Let $(g, G, \textbf{p})_W$ be a real local holomorphic deformation of a real marked map $g$. Assume that one has $|Dg^{q_j-l_j}(c_{l_j,j})|>1$ for all $r<j\le \nu$.
Assume furthermore that all the eigenvalues of $\mathcal{A}$ lie in the set $\{|\rho|\le 1, \rho\not=1\}$.
Then the `positively oriented' transversality condition holds.

\end{coro}
\begin{proof} Write the polynomial $\det (D(\rho))$ in the form $\prod_{i=1}^N (1-\rho \rho_i)$, where $\rho_i\in \C\setminus \{0\}$. If $\rho_i\ge 1$ for some $i$, then $1/\rho_i$ is a zero of $\det(D(\rho))$. As $|1/\rho_i|
\le 1$, $1/\rho_i$ is not an exceptional value. Thus
$\det(I-\mathcal{A}/\rho_i)=0$, which implies that $\rho_i$ is an eigenvalue of $\mathcal{A}$, a contradiction!
\end{proof}

\begin{remark}
Proposition~\ref{prop:non-exceptional} shows that for non-exceptional $\rho$,
one has (\ref{eq:non-exceptional}). One can also associate to $(g,G, \textbf{p})$ another linear operator $\mathcal{A}_J$ for which
\begin{equation}\label{I}
\det D(\rho)=\det (I-\rho \mathcal{A}_J)
\end{equation}
holds for {\em all} $\rho\in \C$.
Here $J$ denotes a collection of all pairs $(i,j)$ such that $1\le j\le \nu$,  $0\le i\le q_j-1$ and if $i=0$ then $j=\mu(j')$ for some $1\le j'\le \nu$. Given a collection of functions $\{c_{i,j}(\lambda)\}_{(i,j)\in J}$ which are holomorphic in a small neighbourhood of $\lambda=0$, there is another collection of holomorphic near $0$ functions $\{\hat c_{i,j}(\lambda)\}_{(i,j)\in J}$ such that
$\hat c_{0,j}(\lambda)=p_j(\textbf{c}_1(\lambda))$ where $\textbf{c}_1(\lambda)=(c_{1,1}(\lambda),\cdots, c_{1,\nu}(\lambda))$
and, for $i\not=0$, $G(\text{c}_1(\lambda), \hat c_{i,j})=c_{i+1,j}(\lambda)$. Here we set
$c_{q_j,j}(\lambda)=c_{0,\mu(j)}(\lambda)$ for $1\le j\le r$ and $c_{q_j,j}(\lambda)=c_{l_j,j}(\lambda)$ for $r<j\le \nu$.
Define the linear map
$\mathcal{A}_J: \C^{\#J}\to \C^{\#J}$ by taking the derivative at $\lambda=0$: $\mathcal{A}_J(\{c_{i,j}'(0)\}_{(i,j)\in J})=\{\hat c_{i,j}'(0)\}_{(i,j)\in J}$.
Explicitely, we get:
$$\hat c'_{i,j}(0)=\left\{
\begin{array}{ll}
\sum_{k=1}^\nu p_{j,k} &\mbox{ if } i=0 \mbox{ and } j=\mu(j') \mbox{ for some } j'\\
\frac{1}{Dg(c_{i,j})} \left(v_{i+1,j}-\sum_{k=1}^\nu L_k(c_{i,j})v_{1,k}\right) & \mbox{ if } 1\le i<q_j-1, 1\le j\le \nu\\
\frac{1}{Dg(c_{q_j-1,j})} \left(v_{0,\mu(j)}-\sum_{k=1}^\nu L_k(c_{q_j-1,j})v_{1,k}\right) &\mbox{ if } i=q_j-1, 1\le j\le r\\
\frac{1}{Dg(c_{q_j-1,j})} \left(v_{l_j,j}-\sum_{k=1}^\nu L_k(c_{q_j-1,j})v_{1,k}\right) & \mbox{ if } i=q_j-1, r< j\le \nu
\end{array}
\right.
$$
Elementary properties of determinants being applied to the matrix $I-\rho \mathcal{A}_J$ lead to~(\ref{I}).
Observe that $\mathcal{A}_J=\mathcal{A}$ if (and only if) all points $c_{i,j}$, $(i,j)\in J$ are pairwise different.
Therefore, we have:
$$\det (I-\rho \mathcal{A})=\det D(\rho)$$
for every $\rho\in \C$
provided $\sum_{j=1}^\nu (q_j-1)+r=\#P$.
\end{remark}

\subsection{The lifting property and the spectrum of $\mathcal A$}\label{subsec:lifting}
We say that the triple $(g,G,\textbf{p})_W$ has {\em the lifting property} if the following holds: Given a holomorphic motion $h_\lambda^{(0)}$ of $g(P)$ over $(\Lambda, 0)$, where $\Lambda$ is a domain in $\C$ which contains $0$, there exist $\eps>0$ and holomorphic motions
$h_\lambda^{(k)}$, $k=1,2,,\cdots$,
of $g(P)$ over $(\D_{\eps}, 0)$
such that
\begin{enumerate}
\item $\textbf{c}_1^{(k)}(\lambda):=h_\lambda^{(k)}(\textbf{c}_1)\in W$ for each $\lambda \in \D_\eps$ and each $k=0,1,2,\dots$;
\item for each $k=0,1,\cdots$, $h_\lambda^{(k+1)}$ is the lift of $h_\lambda^{(k)}$ over $(\D_\eps, \ast)$ for $(g, G, \textbf{p})_W$,
\item there exists $M>0$ such that $|h_\lambda^{(k)}(x)|\le M$ for all
$x\in g(P)$, all $k\ge 0$ and all $\lambda\in \D_{\eps}$.
\end{enumerate}
We say that $(g,G,\textbf{p})_W$ has {\em the weak lifting property} if, for each $h_\lambda^{(0)}$ as above, there exists $\eps>0$ and holomorphic motions $h_\lambda^{(k)}$ of $g(P)$ over $(\D_\eps,0)$, $k=1,2,\cdots$, such that the properties (1) and (2) hold (but we may not have property (3)).

In the case $(g, G, \textbf{p})_W$ is real, we say it has {\em the real lifting property} or {\em the real weak lifting property} if the corresponding property holds for any real holomorphic motions $h_\lambda^{(0)}$.

The following observation is important for us.
\begin{lemma} \label{lem:lift2spectrum}
If $(g,G,\textbf{p})_W$ has the lifting property , then
the spectral radius of the associated transfer operator $\mathcal{A}$ is at most $1$ and every eigenvalue of $\mathcal{A}$ of modulus one is semisimple (i.e. its algebraic multiplicity coincides with its geometric multiplicity). Moreover, for $(g, G, \textbf{p})_W$ real, we only need to assume that the lifting property with respect to real holomorphic motions.
\end{lemma}
\begin{proof} Let us fix an order in the set $g(P)$.
For any $\textbf{v}=(v(x))_{x\in g(P)}$, construct a holomorphic motion $h_\lambda^{(0)}$ over $(\Lambda,0)$ for some domain $\Lambda\ni 0$, such that
$\frac{d}{d\lambda}h^{(0)}_\lambda(x)\left|\right._{\lambda=0}=v(x)$ for all $x\in g(P)$. Then
$$\mathcal{A}^k(\textbf{v})=\left(\frac{d}{d\lambda}h^{(k)}_\lambda(x)\left|\right._{\lambda=0}\right)_{x\in g(P)}$$ for every $k>0$.
By Cauchy's integral formula, there exists $C=C(M, \eps)$ such that $|\frac{d}{d\lambda}h^{(k)}_\lambda(x)\left|\right._{\lambda=0}|\le C$ holds for all $x\in g(P)$ and all $k$. It follows that for any $\textbf{v}\in \C^{\#g(P)}$, $\mathcal{A}^k(\textbf{v})$ is a bounded sequence. Thus the
spectral radius of $\mathcal{A}$ is at most one and every eigenvalue of $\mathcal{A}$ of modulus one is semisimple.

Suppose $(g, G, \textbf{p})_W$ is real. Then for any $\textbf{v}\in \R^{\nu}$, the holomorphic motion $h_\lambda^{(0)}$ can be chosen to be real. Thus if $(g, G,\textbf{p})_W$ has the weak lifting property, then $\{\mathcal{A}^k(\textbf{v})\}_{k=0}^\infty$ is bounded for each $\textbf{v}\in\R^\nu$. The conclusion follows.
\end{proof}

To obtain that the radius is strictly smaller than one, we shall apply the argument to a suitable perturbation of the map $g$.
For example, we have the following:

\begin{lemma} \label{lem:perturb2spectrum}
Let $(g,G,\textbf{p})_W$ be as above.
Let $Q$ be a polynomial such that
$Q(c_{0,j})=0$ for $1\le j\le \nu$ and $Q(x)=0$, $Q'(x)=1$ for every $x\in g(P)$. Let $\varphi_\xi(z)=z-\xi Q(z)$ and for $\xi\in (0,1)$ let
$\psi_{\xi}(\textbf{w})=(\varphi_\xi^{-1}(w_1),\cdots,\varphi_\xi^{-1}(w_\nu))$
be a map from a neighbourhood of $\textbf{c}_1$ into a neighbourhood of $\textbf{c}_1$. Suppose that there exists $\xi\in (0,1)$ such that the triple
$(\varphi_\xi\circ g, \varphi_\xi\circ G_{{\bf u}}, \textbf{p}\circ \psi_\xi)$ has the lifting property erty. Then
the spectral radius of $\mathcal{A}$ is at most $1-\xi$.
\end{lemma}
\begin{proof}
Note that $\tilde{g}:=\varphi_\xi\circ g$ is a marked map with the same sets $P_0\subset P$. Furthermore,
$\tilde{g}^i(c_{0,j})=g^i(c_{0,j})=c_{i,j}$, $D\widetilde{g}(c_{i,j})=(1-\xi) Dg(c_{i,j})$ and
$\frac{\partial \varphi_\xi\circ G}{\partial w_k}({\bf c}_1,z)=\frac{\partial G}{\partial w_k}({\bf c}_1,z)$,
$\frac{\textbf{p}\circ \psi_\xi}{\partial w_k}({\bf c}_1)=(1-\xi)^{-1}\frac{\partial \textbf{p}}{\partial w_k}({\bf c}_1)$.
Therefore, the operator which is associated to the triple $(\varphi_\xi\circ g, \varphi_\xi\circ G_{{\bf u}}, \textbf{p}\circ \psi_\xi)$ is equal to
$(1-\xi)^{-1}\mathcal{A}$,
Since the latter triple has the lifting property , by Lemma~\ref{lem:lift2spectrum},
the spectral radius of $(1-\xi)^{-1}\mathcal{A}$ is at most $1$.
\end{proof}

For completeness we include:

\begin{lemma} Assume that the spectrum radius of $\mathcal{A}$ is strictly less than $1$. Let $h_\lambda^{(k)}$, $k=1,2,,\cdots$, be holomorphic motions
of $P$ all defined over $(\Lambda, 0)$
such that $h_\lambda^{(k+1)}$ is the lift of $h_\lambda^{(k)}$ for each $k$. Assume that $h_\lambda^{(k)}$ are uniformly bounded in $\Lambda$. Then for each $x\in g(P)$, $h_\lambda^{(k)}$ converges to the constant $x$, locally uniformly on $\Lambda$, as $k\to\infty$.
\end{lemma}
\begin{proof} It suffices to prove that there exists $\delta>0$ such that for each $|\lambda|<\delta$, $h_\lambda^{(k)}(x)\to x$ as $k\to\infty$, since we assume that $h^{(k)}_\lambda$ are uniformly bounded in $\Lambda$.

Note that there is a holomorphic map $\Phi=(\varphi_x)_{x\in g(P)}$ from a neighbourhood $V$ of the point
${\bf c}:=g(P)\in \C^{\#g(P)}$ into
$\C^{\#g(P)}$ which fixes ${\bf c}$ such that
$$G_{{\bf z}_1}(\varphi_{x}(Z))=z_{g(x)}, \ \ \ x\in g(P)\setminus P_0, $$
$$\varphi_{c_{0,j}}(Z)=p_j({\bf z}_1), \ \ \ 1\le j\le \nu ,$$
where ${\bf z}_1=(z_{c_{1,j}})_{j=1}^\nu$, $Z=(z_x)_{x\in g(P)}$.
Since the derivative of $\Phi$ at ${\bf c}$ is equal to $\mathcal{A}$, ${\bf c}$ is a hyperbolic attracting fixed point of $\Phi$. Therefore, there exist $r>0$ and $N>0$ such that $\Phi^N$ is defined on the polydisk
$$\mathcal{U}_0=\Pi_{x\in g(P)} B(x, r)$$ 
and maps it compactly into itself. It follow $\Phi^n$ converges uniformly to the constant ${\bf c}$ in $\mathcal{U}_0$.
Since $h^{(k)}_\lambda$ are uniformly bounded, there exists $\delta>0$ such that $h_\lambda^{(k)}(x)\in B(x, r)$ whenever $|\lambda|<\delta$. Since $\Phi$ maps $\textbf{h}^{(k)}_\lambda:
=(h_{\lambda}^{(k)}(x))_{x\in g(P)}$ to $\textbf{h}^{(k+1)}_\lambda$, the statement follows.
\end{proof}
\subsection{If $\mathcal A$ has eigenvalue one then $\mathcal R(w)=0$ is a variety}\label{subsec:Aeig1}
\begin{theorem}\label{thm:1eigen} Assume that either the triple $(g, G, \textbf{p})_W$ has the lifting property or $(g, G,\textbf{p})_W$ is real and has the real lifting property.
Then we have the following alternative:
\begin{enumerate}
\item All eigenvalues of $\mathcal{A}$ are contained in $\overline{\D}\setminus\{1\}$;
\item There exists a neighbourhood $W'\subset W$ of ${\bf c}_1$
such that
\begin{equation}\label{allornoth}
\{\textbf{w}\in W'  \, | \,  \mathcal{R}(\textbf{w})=0\}
\end{equation}
is an analytic variety of (complex) dimension at least $1$.
\end{enumerate}
\end{theorem}

If the second alternative holds,  then there is a local analytic manifold through ${\bf c}_1$ such that $\mathcal{R}(\textbf{w})=0$ holds identically in this manifold.
If $\nu=1$, the manifold must contain a neighbourhood of ${\bf c}_1$, in other words, $\mathcal{R}(\textbf{w})=0$ holds for every ${\bf w}\in \C$ near ${\bf c}_1\in \C$. Note that in most situations, one can easily show that the second alternative is invalid. Thus this theorem is very useful to obtain transversality.

Let $\Lambda$ be a domain in $\C$ which contains $0$.
A holomorphic motion $h_\lambda(x)$ of $g(P)$ over $(\Lambda,0)$ is called {\em asymptotically invariant of order $m$} (with respect to $(g,G,\textbf{p})_W$) if there is a subdomain $\Lambda_0\subset \Lambda$ which contains $0$ and a holomorphic motion  $\widehat{h}_\lambda(x)$ which is the lift of $h_\lambda$ over $(\Lambda_0,0)$, such that
\begin{equation}
\widehat{h}_\lambda(x)-h_\lambda(x)=o(\lambda^{m+1})\text{ as } \lambda\to 0.
\label{eq:asymptm}\end{equation}
Obviously,
\begin{lemma} $1$ is an eigenvalue of $\mathcal{A}$ if and only if there is a non-degenerate holomorphic motion which is invariant of order $1$.
\end{lemma}
Here, a holomorphic motion $h_\lambda(x)$ is called {\em non-degenerate} if $\frac{d}{d\lambda}h_\lambda(x)\left|\right._{\lambda=\ast}\not=0$ holds for some $x\in P$.

A crucial step in proving this theorem is the following Lemma~\ref{lem:arbitraryasyminv} whose proof requires the following easy fact. 
\begin{fact} Let $F:\mathcal{U}\to \C$ be a holomorphic function defined in an open set $\mathcal{U}$ of $\C^N$, $N\ge 1$. Let $\gamma, \widetilde{\gamma}:\D_\eps\to \mathcal{U}$ be two holomorphic curve such that $$\gamma(\lambda)-\widetilde{\gamma}(\lambda)=O(\lambda^{m+1})\text{ as }\lambda\to 0.$$
Then
$$F(\gamma(\lambda))-F(\widetilde{\gamma}(\lambda))=\sum_{i=1}^N \frac{\partial F}{\partial z_i}(\gamma(0))(\gamma_i(\lambda)-\widetilde{\gamma}_i(\lambda))+O(\lambda^{m+2})\text{ as }\lambda\to 0.$$
\end{fact}
\begin{proof} For fixed $\lambda$ small, define $\delta(t)=(1-t)\widetilde{\gamma}(\lambda)+t\gamma(\lambda)$ and $f(t)=F(\delta(t))$. Then
$$f'(t)=\sum_{i=1}^N \frac{\partial F}{\partial z_i}(\delta(t)) (\gamma_i(\lambda)-\widetilde{\gamma_i}(\lambda)).$$
Since $\delta(t)-\gamma(0)=O(\lambda)$, and
$F(\gamma(\lambda))-F(\widetilde{\gamma}(\lambda))=\int_0^1 f'(t)dt$, the equality follows.
\end{proof}
\begin{lemma}\label{lem:arbitraryasyminv}
One has the following:
\begin{enumerate}
\item Assume $(g, G,\textbf{p})_W$ has the lift property.
Suppose that there is a non-degenerate holomorphic motion $h_\lambda$ of
$g(P)$ over $(\Lambda, 0)$ which is asymptotically invariant of order $m$ for some $m\ge 1$. Then there is a non-degenerate holomorphic motion $H_\lambda$ of $g(P)$ over some $(\tilde\Lambda, 0)$ which is asymptotically invariant of order $m+1$. Besides, $H_\lambda(x)-h_\lambda(x)=o(\lambda^{m+1})$ as $\lambda\to 0$ for all $x\in g(P)$.
\item Assume $(g, G,\textbf{p})_W$ is real and has the real lift property. Suppose that
there is a non-degenerate real holomorphic motion $h_\lambda$ of
$g(P)$ over $(\Lambda, 0)$ which is asymptotically invariant of order $m$ for some $m\ge 1$. Then there is a non-degenerate real holomorphic motion $H_\lambda$ of $g(P)$ over some $(\tilde\Lambda, 0)$ which is asymptotically invariant of order $m+1$. Besides, $H_\lambda(x)-h_\lambda(x)=o(\lambda^{m+1})$ as $\lambda\to 0$ for all $x\in g(P)$.
\end{enumerate}
\end{lemma}
\begin{proof} We shall only prove the first statement as the proof of the second is the same with obvious change of terminology. Let $h_\lambda$ be a non-degenerate holomorphic motion of
$P$ over $(\Lambda, 0)$  which is asymptotically invariant of order $m$.
%
By assumption that $(g, G, \textbf{p})_W$ has the lifting property , there exists a smaller domain $\Lambda_0\subset \Lambda$ and holomorphic motions $h^{(k)}_\lambda$ over $\Lambda_0$, $k=0,1,\ldots$ such that
$h^{(0)}_\lambda=h_\lambda$ and such that $h^{(k+1)}_\lambda$ is the lift of
$h^{(k)}_\lambda$ over $(\Lambda_0,0)$ for each $k\ge 0$. Moreover, the functions $h^{(k)}_\lambda$ are uniformly bounded. For each $k\ge 1$, define
$$\psi_\lambda^{(k)}(x)=\frac{1}{k}\sum_{i=0}^{k-1}h_\lambda^{(i)}(x),$$
and $$\varphi_\lambda^{(k)}(x)=\frac{1}{k}\sum_{i=1}^k h_\lambda^{(i)}(x).$$
By shrinking $\Lambda_0$, we may assume that there exists $k_n\to\infty$, such that $\psi_\lambda^{(k_n)}(x)$ converges uniformly in $\lambda\in\Lambda_0$ as $k_n\to\infty$ to a holomorphic function $H_\lambda(x)$. Shrinking $\Lambda_0$ furthermore if necessary, $H_\lambda$ defines a holomorphic motion of $g(P)$ over $(\Lambda_0,0)$. Clearly, $\varphi_\lambda^{(k_n)}(x)$ converges uniformly to $H_\lambda(x)$ as well.

Let us show that $H_\lambda$ is asymptotically invariant of order $m+1$ by applying the fact above. This amounts to show
\begin{enumerate}
\item [(i)] For each $x\in g(P)\setminus P_0$, and any $k\ge 1$,
$$G_{\psi^{(k)}_\lambda(c_{1,1}),\cdots, \psi_\lambda^{(k)}(c_{1,\nu})}(\varphi_\lambda^{(k)}(x))=\psi_\lambda^{(k)}(x) +O(\lambda^{m+2}) \text{ as } \lambda\to 0.$$
\item [(ii)] For $x=c_{0,j}\in g(P)$,
$$p_{j}(\psi^{(k)}_\lambda(c_{1,1},\cdots, \psi_\lambda^{(k)}(c_{1,\nu}))=\varphi^{(k)}_\lambda (c_{0,j})+ O(\lambda^{m+2}) \text{ as }\lambda\to 0.$$
\end{enumerate}

Let us prove (i). Fix $x\in g(P)\setminus P_0$ and $k\ge 1$. Let $F(z_1,z_2,\cdots, z_\nu, z_{\nu+1})=G_{(z_1,z_2,\cdots, z_\nu)}(z_{\nu+1})$. By the construction of $h^{(k)}$,
we have
$$F(h_\lambda^{(i)}(c_{1,1}),\cdots, h_\lambda^{(i)}(c_{1,\nu}), h_\lambda^{i+1}(x))=h_\lambda^{(i)}(g(x))$$
for every $i\ge 0$. Thus
\begin{equation}\label{eqn:arbitraryasyminv1}
\psi^{(k)}_\lambda(g(x))=\frac{1}{k} \sum_{i=0}^{k-1}F(h_\lambda^{(i)}(c_{1,1}),\cdots, h_\lambda^{(i)}(c_{1,\nu}), h_\lambda^{i+1}(x)).
\end{equation}
Since all the functions $h^{(i)}(x)$,$\psi^{(k)}_\lambda(x)$, $\varphi^{(k)}_\lambda(x)$ have the same derivatives up to order $m$ at $\lambda=0$, applying Fact, we obtain
\begin{multline*}
F(h_\lambda^{(i)}(c_{1,1}),\cdots, h_\lambda^{(i)}(c_{1,\nu}), h_\lambda^{i+1}(x))-F(\psi_\lambda^{(k)}(c_{1,1}),\cdots, \psi^{(k)}_\lambda(c_{1,\nu}),\varphi^{(k)}_\lambda(x))\\
=\sum_{j=1}^{\nu}\frac{\partial F}{\partial z_j}(\textbf{c}_1,x) (h_\lambda^{(i)}(c_{1,j})-\psi_\lambda^{(k)}(c_{1,j})+\frac{\partial F}{\partial z_{\nu+1}}(\textbf{c}_1, x) (h_\lambda^{(i+1)}(x)-\varphi^{(k)}_\lambda(x))+O(\lambda^{m+2}),
\end{multline*}
as $\lambda\to 0$. Summing over $i=0,1,\ldots, k-1$, we obtain
\begin{multline*}
\frac{1}{k} \sum_{i=0}^{k-1}F(\varphi_\lambda^{(i)}(c_{1,1}),\cdots, \varphi_\lambda^{(i)}(c_{1,\nu}), h_\lambda^{i+1}(x))\\
=F(\psi^{(k)}_\lambda(c_{1,1}),\cdots, \psi^{(k)}(c_{1,\nu}), \varphi^{(k)}_\lambda(x))+ O(\lambda^{m+2}).
\end{multline*}
Together with (\ref{eqn:arbitraryasyminv1}), this implies the equality in (i).

For (ii), we  use $F(z_1,\cdots, z_\nu)=p_j(z_1, \cdots, z_\nu)$ and argue in a similar way.
\end{proof}

\begin{proof}[Proof of Theorem~\ref{thm:1eigen}] By Lemma~\ref{lem:lift2spectrum}, all eigenvalues of $\mathcal{A}$ are contained in $\overline{\D}$. Assume that (1) fails. Then $1$ is an eigenvalue of $\mathcal{A}$. Let $\textbf{v}=(v(x))_{x\in g(P)}$ be its eigenvector. In the case that $(g,G,\textbf{p})_W$ is real, we choose $\textbf{v}$ to be real. Then $h_\lambda(x)=x+ v(x) \lambda$ defines a non-degenerate holomorphic motion of $P$ over some $(\Lambda, 0)$ and this holomorphic motion is asymptotically invariant of order $1$. By Lemma~\ref{lem:arbitraryasyminv}, it follows that for each $m\ge 1$, there is a holomorphic motion
$a^{(m)}_\lambda$ of $P$ which is asymptotically invariant of order $m$
and such that $D_0 a^{(m)}_\lambda(x)=v(x)$. It follows that for every $m$,
for each
$1\le j\le r$, $G^{q_j-1}_{{\bf a}^{(m)}_1(\lambda)}(a^{(m)}_\lambda(c_{1,j}))=p_{\mu(j)}({\bf a}^{(m)}_1(\lambda))+o(\lambda^m)$ and
for $r+1\le j\le \nu$, $G^{q_j-1}_{{\bf a}^{(m)}_1(\lambda)}(a^{(m)}_\lambda(c_{1,j}))=G^{l_j-1}_{{\bf a}^{(m)}_1(\lambda)}(a^{(m)}_\lambda(c_{1,j}))+o(\lambda^m)$.
Here ${\bf a}^m_{1}(\lambda)=(a^{(m)}_\lambda(c_{1,1}),\cdots,a^{(m)}_\lambda(c_{1,\nu}))$.
Recall that a map $\mathcal{R}=(R_1,\cdots,R_\nu)$ from a neighbourhood of the point ${\bf c}_1=(c_{1,1},\cdots,c_{1,\nu})\in \C^\nu$ into $\C^\nu$ is defined as follows:
$$
\left\{\begin{array}{ll}
R_j({\bf w})=G^{q_j-1}_{{\bf w}}(w_j)-p_{\mu(j)}({\bf w}), \ \ \  1\le j\le r,\\
R_j({\bf w})=G^{q_j-1}_{{\bf w}}(w_j)-G^{l_j-1}_{{\bf w}}(w_j), \ \ \  r+1\le j\le \nu.
\end{array}
\right.
$$
where ${\bf w}=(w_1, w_2,...,w_\nu)$.
Then for every $m>0$,
$\mathcal{R}({\bf a}^{(m)}_{1}(\lambda))={\bf \psi}^{(m)}(\lambda)$
where
${\bf \psi}^{(m)}(\lambda)=(\psi^{(m)}_1(\lambda),\cdots,\psi^{(m)}_\nu(\lambda))$ and
$\psi^{(m)}_j(\lambda)=o(\lambda^m)$, $j=1,\cdots,\nu$.
Also, ${(\bf a}^m_{1})'(0)={\bf v}\not={\bf 0}$.
Now we assume the contrary: the local analytic variety defined by the equation $\mathcal{R}({\bf w})=0$ and containing ${\bf c}_1$ is $0$-dimensional, i.e., consists of a single point ${\bf c}_1$. By general properties of analytic varieties,
${\bf c}_1$ is then an isolated zero of $\mathcal{R}^{-1}({\bf 0})$. There exists some $k\ge 1$ such that for any point ${\bf u}$ which is close to ${\bf 0}$ and outside an analytic variety of dimension less than $\nu$
 (on which the Jacobian of the map $\mathcal{R}$ is equal to zero) the equation
$\mathcal{R}({\bf w})={\bf u}$ has precisely $k$ different solutions ${\bf w}^i({\bf u})=(w^i_1({\bf u}),\cdots,w^i_\nu({\bf u}))$, $i=1,\cdots,k$.
It follows that for every coordinate $j$ the following  function: $P_j(z, {\bf u}):=\Pi_{i=1}^k(z-w^i_j({\bf u}))$
extends to an analytic function in ${\bf u}$ in a neighbourhood $W'$ of ${\bf 0}$.
Thus $$P_j(z,\textbf{u})=P_j(z,0)+O(\|\textbf{u}\|)=(z-c_{1,j})^k+ O(\|\textbf{u}\|) \text{ as } \textbf{u}\to \textbf{0}.$$
For each $m\ge 1$, $P_j(a^{(m)}_\lambda(c_{1,j}), {\bf \psi}^{(m)}(\lambda))=0$ holds for every $\lambda$ near $0$. Therefore,
$$(a^{(m)}_\lambda(c_{1,j})-c_{1,j})^k=O(\|\psi^{(m)}(\lambda)\|)=O(\lambda^m) \text{ as } \lambda\to 0 $$
holds for every $j$ and $m$.
Taking $j$ such that $v(c_{1,j})\not =0$ and $m>k$, we obtain
$$(\lambda v(c_{1,j})+o(\lambda))^k=O(\lambda^m)\text{ as }\lambda\to 0,$$
which is absurd.
\end{proof}


\part*{Part B: Application to covering maps, and in particular polynomial-like maps,  satisfying a separation property}
In the next three sections we will show that the method developed in Part A
works well in the case of covering maps, and in particular polynomial-like maps,  satisfying some separation property.
\section{Families of the form $f_\lambda(x)=f(x)+\lambda$ and $f_\lambda(x)=\lambda f(x)$}
\label{sec:oneparameter}
In this section we will  apply these techniques to show that one has monotonicity and the transversality properties (\ref{eq:trans}) and (\ref{eq:trans2}) within certain families
of real real maps of the form $f_\lambda(x)=f(x)+\lambda$ and $f_\lambda(x)=\lambda \cdot f(x)$
where $x\mapsto f(x)$ has one critical value (and is unimodal - possibly on a subset $\R$) or satisfy symmetries.  There are quite a few
papers giving examples for which one has non-monotonicity for such families,  see for example \cite{Br,Ko,NY,Zd}.
In this section we will prove several theorems which show monotonicity for a fairly wide class
of such families.

Note that one can only expect transversality in one-parameter families for which one either has precisely one singular value
or for which one has two singular values but for which one has additional symmetry.
Indeed, it is easy to construct a one-parameter family of bimodal maps for which transversality fails, see \cite{Str2}.

In Subsection~\ref{subsec:additive} we show that the methods we developed in the previous section
apply if one has something like a polynomial-like map $f\colon U\to V$  with sufficiently  {\lq}big complex bounds{\rq}. This gives yet another proof for monotonicity
for real families of the form $z^\ell+c$, $c\in \R$ in the setting when $\ell$ is an even integer.
 We also apply this method to a family of maps with a flat critical point
in Subsection~\ref{subsec:flatcritical}. In Subsection~\ref{subsec:multiplicative} we show how to obtain
the lifting property in the setting of one parameter families of entire maps.

\subsection{Families of the form $f_\lambda(x)=f(x)+\lambda$ with a single critical point}\label{subsec:additive}
Consider a marked map $g$ from a finite set
$P$ into itself with $P\supset P_0=\{0\}$ and define a local holomorphic deformation $(g,G,\textbf{p})$ of $g$ as follows: $G_w(z)=g(z)+(w-g(0))$ and $\textbf{p}(w)=0$ for $w$ in a neighbourhood $W$ of $c_1=g(0)\in \C$.
\begin{theorem}\label{single}
Suppose that $g$ extends to a holomorphic map $g: U_g\to V$ where
\begin{itemize}
\item $U_g$ is a bounded open set in $\C$ such that $U\supset P\setminus \{0\}$ and $0\in \overline{U}_g$;
\item $V$ is a bounded open set in $\C$ such that $c_1:=g(0)\in V$;
\item $g: U_g\setminus \{0\}\to V\setminus \{c_1\}$ is an unbranched covering.
\end{itemize}
If the separation property
\begin{equation}\label{1exp}
V\supset B(c_1;\diam(U_g)) \supset U_g.
\end{equation}
holds
then the spectrum of the operator $\mathcal{A}$ is contained in $\overline{\D}\setminus \{1\}$.
If the robust separation property
\begin{equation}\label{1robust}
V\supset \overline{B(c_1;\diam (U_g))} \supset U_g
\end{equation}
holds, then the spectral radius of $\mathcal{A}$ is strictly smaller than $1$.
In particular, if $g^q(0)=0$, then
$$\det (I-\rho \mathcal{A})=\sum_{i=0}^{q-1}\frac{\rho^i}{Dg^i(c_1)}\not=0$$
holds for all $|\rho|\le 1$.
\end{theorem}

\begin{proof}
Let $W=U_g$. Let us show that (\ref{1exp}) implies that
$(g, G, \textbf{p})_W$ has the lift property.

For each domain $\Delta\ni 0$ in $\C$, let $\mathcal{M}_\Delta$ denote the collection of all holomorphic motions $h_\lambda$ of $g(P)$ over $(\Delta,0)$ such that
\begin{equation}
h_\lambda(x)\in U\text{ for all }x\in g(P)\text{ and }\lambda\in \Delta.
\end{equation}

{\bf Claim.} Let $\Delta\ni 0$ be a simply connected domain in $\C$. Any holomorphic motion $h_\Delta$ in $\mathcal{M}_\Delta$ has a lift $\widehat{h}_\lambda$ which is again in the class $\mathcal{M}_\Delta$.

Indeed, for each $x\in P\setminus\{0\}$, and any $\lambda\in\Delta$, $$0<|h_\lambda(g(x))-h_\lambda(g(0))|<\diam (U_g),$$
hence by (\ref{1exp}),
$$h_\lambda(g(x))-h_\lambda(g(0))+g(0)\in V\setminus \{g(0)\}.$$
Since $g:U_g\setminus\{0\}\to V\setminus \{g(0)\}$ is an unbranched covering and $\Delta$ is simply connected, there is a holomorphic function $\lambda\mapsto \widehat{h}_\lambda(x)$, from $\Delta$ to $U_g\setminus\{0\}$, such that $\widehat{h}_0(x)=x$ and
$$g(\widehat{h}_\lambda(x))=h_\lambda(g(x))-h_\lambda(g(0))+g(0),$$
i.e.,
$$G_{h_\lambda(g(0))}(\widehat{h}_\lambda(x))=h_\lambda(g(x)).$$
Define $\widehat{h}_\lambda(0)=0$ if $0\in g(P)$. Then $\widehat{h}_\lambda$ is a lift of $h_\lambda$ over $\Delta$.

For any holomorphic motion $h_\lambda$ of $g(P)$ over $(\Lambda,0)$, there is a simply connected sub-domain $\Delta\ni 0$ such that the restriction of $h_\lambda$ on $\Delta$ belongs to the class $\mathcal{M}_{\Delta}$. It follows that $(g,G,\textbf{p})_W$ has the lift property.

Therefore the assumptions of
Theorem~\ref{thm:1eigen}  are satisfied.
The second alternative of Theorem~\ref{thm:1eigen} cannot hold, because otherwise for
all parameters $w\in W$ the $G_w$ would have the same dynamics. Hence, the conclusion follows.

If the robust separation property~(\ref{1robust}) holds, then
Lemma~\ref{lem:perturb2spectrum} applies and therefore the spectral radius of $\mathcal{A}$ is strictly smaller than $1$.
\end{proof}

\begin{coro}\label{cor:zd}
For any even integer $d$, transversality condition (\ref{eq:trans2}) holds and  the topological entropy of $g_c(z)=z^d+c$ depends monotonically on $c\in \R$.
\end{coro}

\subsection{A unimodal family with a flat critical  point: Proof of Theorem~\ref{thm:flat}}
\label{subsec:flatcritical}
Fix $\ell\ge 1$, $b> 2(e\ell)^{1/\ell}$ and consider
$$f_c(x)=\left\{ \begin{array}{rl} b e^{-1/|x|^\ell} +c  &\mbox{ for }x\in \R\setminus \{ 0\}, \\
c & \mbox{ for }x=0 .\end{array} \right.
.$$
The assumption on $b$ implies that $b= 2x e^{1/x^\ell}$ has a solution $x=\beta\in (0, \ell^{1/\ell})$.
This means that the map $f_{-\beta}$ has the Chebeshev combinatorics: $f_{-\beta}(0)=-\beta$ and $f_{-\beta}(\beta)=\beta$.
Note that
$$D f_{-\beta}(\beta)= be^{-1/\beta^\ell} \frac{\ell }{\beta^{\ell+1}}= \frac{2\ell}{\beta^\ell}>2.$$
Therefore, there exists $x_1>x_0>\beta$ such that
$f_{-\beta}(x_0)=x_1$ and $x_1-\beta> 2(x_0-\beta)$.
Choosing $x_0$ close enough to $\beta$, we have
$$R:=f_0(x_0)=x_1+\beta<b.$$

For a bounded open interval $J\subset \R$, let $D_*(J)$ denote the Euclidean disk with $J$ as a diameter.
\begin{lemma}\label{lem:flatext} The map $f_0: (-x_0,0)\cup (0,x_0)\to (0, R)$ extends to an unbranched holomorphic covering map $F_0:U\to B^*(0,R)$, where
$U\subset D_*((-x_0,0))\cup D_*((0,x_0))$.
In particular, $\diam (U)=2x_0 <R$.
\end{lemma}
\begin{proof} Let $\Phi(re^{i\theta})= r^\ell e^{i\ell \theta}$ denote the conformal map from the sector $\{re^{i\theta}: |\theta|<\pi/(2\ell)\}$ onto the right half plane, let $U^+=\Phi^{-1}(D_*((0, x_0^\ell)))$. As usual, the Schwarz Lemma, implies $U^+\subset D_*((0,x_0))$.  Define $U^-=\{-z: z\in U^+\}$, $U=U^+\cup U^-$ and
$$F_0(z)=\left\{\begin{array}{ll}
b e^{-1/\Phi(z)} & \mbox{ if } z\in U^+\\
be^{-1/\Phi(-z)} & \mbox{ if } z\in U^-.
\end{array}
\right.
$$
It is straightforward to check that $F_0$ maps $U^+$ (resp. $U_-$) onto $B^*(0,R)$ as an un-branched covering.
\end{proof}

\begin{proof}[Proof of Theorem~\ref{thm:flat}]
Let $c\ge -\beta$ be such that $f_c^q(0)=0$ for some $q\ge 1$. We need to show that $\sum_{n=0}^{q-1}1/Df_c^i(c)>0$.
So it suffices to show that the assumptions of  Theorem~\ref{single} are satisfied.
To see this, let  $F_0\colon U\to B^*(0,R)$ be the map given by Lemma~\ref{lem:flatext},  and define $F_c=F_0+c$ and $V=B(c, R)$.
Since $-\beta\le c\le 0$, we have that $P:=\{f_c^i(0): 1\le i<q\}\subset U$. Moreover,
$F_c\colon U\to V\setminus \{c\}$ is an  unbranched covering. Finally,  Lemma~\ref{lem:flatext} gives $\diam(U)<R$ and
thus we obtain that
(\ref{1robust}) holds.
\end{proof}

\subsection{Families of the form $f_a(x)=af(x)$}\label{subsec:multiplicative}

There are quite a few papers which ask the question:
\begin{quote}For which interval maps $f$, has one monotonicity of the entropy
for the family  $x\mapsto af(x)$?
\end{quote}
This question is subtle, as the
counter examples to various conjectures show, see \cite{NY, Ko, Br, Zd}.
In this section we will obtain monotonicity and transversality for a very large class of maps $f$.

As usual we say that $v\in \C$ is a {\em singular value} of a holomorphic map $f: D\to \C$ if it is a critical value,
or an asymptotic value where the latter means the existence of a path $\gamma: [0,1)\to D$ so that $\gamma(t)\to \partial D$ and $f(\gamma(t))\to v$ as $t\to 1-$.

Consider holomorphic maps $f: D\to \C$ such that:
\begin{enumerate}
\item[(a)] $D$ is a domain which is symmetric w.r.t. $\R$ and $D\cap \R=I$ where $I$ is an open interval (finite or infinite), $0\in I$;
\item[(b)] $f(I)\subset \R$, $f(D)=\C$ and the only possible (finite) asymptotic value of $f$ is $0$;
\item[(c)] $f(0)=0$;
\end{enumerate}
Let $\mathcal{E}$ be the class of maps which satisfy $(a)$,$(b)$,$(c)$ and assumption $(d)$:
\begin{enumerate}
\item[(d)] the only critical values of $f$ are $1$ and, perhaps, $0$ and there exists a  minimal $c>0$ such that $f$ has a positive local maximum at $c$.
\end{enumerate}
Similarly let  $\mathcal{E}_o$ be the class of maps which satisfy $(a)$,$(b)$,$(c)$ and assumption $(e)$:
\begin{enumerate}
\item[(e)] $f$ is odd, the only critical values of $f$ are $\pm 1$ and, perhaps, $0$ and there exists a minimal $c>0$ such that $f$ has a positive local maximum at $c$.
\end{enumerate}

Classes $\mathcal{E}$ and $\mathcal{E}_o$ are rich even in the case $D=\C$. See \cite{GO} for a general method of constructing entire (or meromorphic) functions with prescribed asymptotic and critical values.
These classes are also non-empty if the domain $D$ is a topological disk or even if $D$ not simply-connected \cite{Er}.

Note that for $f$ in class $\mathcal E\cup \mathcal{E}_0$, $f(c)=1$. Put $b=\sup\{x\in I: x>c, f(x)>0\}$. Then $f'(x)<0$ holds for $x\in (c, b)$ and thus $f: [0,b)\to \R$ is unimodal.
Let us call $J=[0,b)$ the {\em unimodal part} of $f$.

Examples of entire functions $f$ of the class $\mathcal{E}$
are
\begin{itemize}
\item $f(z)=4z(1-z)$,
\item  $f(z)=4\exp(z)(1-\exp(z))$,
\item $f(z)=[\sin(z)]^2$,
\item  $f(z)=m^{-m} (ez)^m\exp(-z)$ when $m$ is a positive even integer.
\end{itemize}
Examples  of maps in the class $\mathcal{E}_o$ are
\begin{itemize}
\item $f(z)=\sin(z)$ and
\item $f(z)= (m/2)^{-m/2}e^{m/2}z^m \exp(-z^2)$ when $m$ is a positive odd integer.
(in fact, $z^m\exp(-z^2)$ and $z^m\exp(-z)$ are conjugated by $z\mapsto 2z^2$).
\end{itemize}

Using qs-rigidity, it was already shown in \cite{RS} that the topological entropy of $\R\ni x \mapsto af(x)$
is monotone $a$, where $f(x)=\sin(x)$ or more generally $f$ is real, unimodal and entire on the complex plane and satisfies a certain sector condition.
Here we strengthen and generalise this result as follows:

\begin{theorem}\label{thm:classE}
Let $f$ be either in $\mathcal{E}$ or in $\mathcal{E}_o$. Assume that the critical point $c>0$ is either periodic or eventually periodic for $f_a(x)=af(x)$ where $0<a<b$.
Then the transversality properties (\ref{eq:trans}) resp.
(\ref{eq:trans2}) hold for the family $f_a(x)$.
In particular, the kneading sequence of the  family $f_a(x): J\to \R$ is 
monotone increasing. 
\end{theorem}
\begin{proof}
The proof is an easy application of Theorem~\ref{thm:holo}.
Let $f\in \mathcal{E}\cup \mathcal{E}_o$. Denote $g(x)=af(x)$.
Let $P_0=\{c\}$, $P=\{c_i=g^i(c): i\ge 0\}$.
By the assumptions, $g$ extends to a holomorphic map $g: D\to \C$, $g(P)\subset P$ and $Dg(x)\not=0$ for any $x\in P\setminus P_0$. In particular, $g$ is an extension of a real marked map.
For each $w\in W:=\C_*=\C\setminus \{0\}$, $G_w(z):=wf(z)$ is a branched covering from $U=D\setminus \{0\}$ onto $V=\C_*$. Define $p(w)\equiv c$. Then $(g, G, p)_W$ is an (extension) of a local holomorphic deformation of $g$. It suffices to prove that $(g, G, p)_W$ has the lift property so that Theorem~\ref{thm:1eigen} applies. Note that the second alternative in the conclusion of Theorem~\ref{thm:1eigen} clearly fails in our setting.

Let us first consider the case $f\in \mathcal{E}$. In this case, $w$ is the only critical value of $G_w$. Given a simply connected domain $\Delta\ni 0$ in $\C$, let $\mathcal{M}_\Delta$ denote the collection of all holomorphic motions $h_\lambda$ of $g(P)$ over $(\Delta,0)$ with the following properties: $h_\lambda(x)\in U$ for all $x\in g(P)$ and $\lambda\in \Delta$. Given such a holomorphic motion, for each $x\in g(P)$ there is a holomorphic map $\lambda\mapsto \widehat{h}_\lambda(x)$, $\lambda\in\Delta$, with $\widehat{h}_0(x)=x$ and such that
$f(\widehat{h}_\lambda(x))=h_\lambda(f(x))/h_\lambda(g(c))$. Indeed, for $x=c$, takes $\widehat{h}_\lambda(x)\equiv c$ and for $x\in g(P)\setminus\{c\}$, we have $h_\lambda(f(x))/h_\lambda(g(c))\in V\setminus \{1\}$ so the existence of $\widehat{h}_\lambda$ follows from the fact that
$f: U\setminus f^{-1}(1)\to V\setminus \{1\}$ is an unbranched covering.  Clearly, $\widehat{h}_\lambda$ is a holomorphic motion in $\mathcal{M}_\Delta$ and it is a lift of $h_\lambda$ over $\Delta$.
It follows that $(g, G, p)_W$ has the lift property. Indeed, if $h_\lambda$ is a holomorphic motion of $g(P)$ over $(\Lambda, 0)$ for some domain $\Lambda\ni 0$ in $\C$, then we can take a small disk $\Delta\ni 0$ such that the restriction of $h_\lambda$ on $(\Delta, 0)$ is in the class $\mathcal{M}_\Delta$. Therefore, there exists a sequence of holomorphic motions $h_\lambda^{(k)}$ of $g(P)$ over $(\Delta,0)$ such that $h_\lambda^{(0)}=h_\lambda$ and $h_\lambda^{(k+1)}$ is a lift of $h_\lambda^{(k)}$ over $\Delta$ for each $k\ge 0$. If $x=c$ then $h_\lambda^{(k)}(x)\equiv c$ for each $k\ge 1$ while if $x\in g(P)\setminus \{c\}$, $h_\lambda^{(k)}(x)$ avoids values $0$ and $c$. Restricting to a small disk, we conclude by Montel's theorem that $\lambda \mapsto h_\lambda^{(k)}(x)$ is bounded.

The case $f\in\mathcal{E}^o$ is similar. In this case, $G_w$ has two critical values $w$ and $-w$, but it has additional symmetry being an odd function. Given a simply connected domain $\Delta\ni 0$ in $\C$, let $\mathcal{M}_\Delta^o$ denote the collection of all holomorphic motions $h_\lambda$ of $g(P)$ over $(\Delta,0)$ with the following properties: for each $\lambda\in \Delta$,
\begin{itemize}
\item $h_\lambda(x)\in U$ for all $x\in g(P)$;
\item $h_\lambda(x)\not=-h_\lambda(y)$ for $x, y\in g(P)$ and $x\not=y$.
\end{itemize}
Then similar as above, we show that each $h_\lambda$ in $\mathcal{M}_\Delta^o$ has a lift which is again in the class $\mathcal{M}_\Delta^o$. It follows that $(g, G, p)_W$ has the lift property.
\end{proof}

\section{The separation property in a more general setting}\label{sec:separation}
In this section we present a unified set-up to treat examples like in Section~\ref{sec:oneparameter}. We give some applications of this quite general scheme later in the paper.

\subsection{Holomorphic covering maps and their deformations}\label{subsec:localdefn}
We use $\HH$ to denote all the triples $(g, C(g), \textbf{v})$ where $g:U_g\to V_g\subset \C$ is a holomorphic map,  $C(g)$ is a discrete subset of $\overline{U_g}$, $\textbf{v}\in \C^\nu$ for some $\nu\in \{1,2,\ldots\}$,  such that
\begin{itemize}
\item $U_g$ is a finite union of pairwise disjoint open sets $U_g^{(i)}$ in $\C$, $i=1,2,\ldots,m$ and $V_g$ is a union of open sets $V_g^{(i)}$;
\item for each $i$, $g: U_g^{(i)}\to V_g^{(i)}:=g(U_g^{(i)})$ is non-constant,
and
$g: U_g^{(i)}\setminus g^{-1}(C^v(g))\to V_g^{(i)}\setminus C^v(g)$ is an un-branched covering, where $C^v(g)$ is the set consisting of all coordinates of $\textbf{v}$;
\item
\begin{equation}\label{ccv}
\{c\in U_g: g'(c)=0\}\subset C(g)\cap U_g \subset g^{-1}(C^v(g)).
\end{equation}
\end{itemize}
We shall call such a map $g$ a {\em holomorphic branched covering}, $C(g)$ the {\em singular point set} and $\textbf{v}$ the {\em singular value vector}.
Note that here $V_g$ can be a bounded subset of $\C$.

A {\em local holomorphic deformation}
of $(g, C(g),\textbf{v})$ is a triple $(g, G, \textbf{p})_W$, where $W$ is a neighbourhood of $\textbf{v}$ in $\C^\nu$ and
\begin{enumerate}[topsep=0.1cm,itemsep=0.05ex,leftmargin=0.6cm]
\item $\textbf{p}=\{p_k\}_{k=1}^N$ where $N=\# C(g)$, each map $p_k:W\to \C$ is holomorphic and furthermore, $p_k(w)\not=p_{k'}(w)$ for all $w\in W$ and all $k\not=k'$;
\item there exists a positive integer $m$ and for each $w\in W$, there is a holomorphic map $G_w: U_w\to \C$, where $U_w$ is a union of pairwise disjoint non-empty
open sets $U_w^{(i)}$ in $\C$, $i=1,2,\ldots, m$, such that $(G_w, C(G_w),w)\in\HH$, where $C(G_w)=\{p_j(w): 1\le j\le N\}$;
\item $G_{\textbf{v}}=g$ and $C(g)=C(G_{\textbf{v}})$;
\item \label{item4}
If $p_j(\textbf{v})\in U_g$ then $p_j(w)\in U_w$ for all $w\in W$ and if $p_j(\textbf{v})\not\in U_g$ then $p_j(w)\not\in U_w$ for all $w\in W$;
\item For each $i=1,2,\ldots, m$, $\mathcal{U}^{(i)}=\{(w,z)\in \C^{\nu+1}: w\in W, z\in U_w^{(i)}$ is open in $\C^{\nu+1}$. Defining the map $G(w,z)= (w, G_w(z)$,
    and
$$C^v(G)=\{(w_1,w_2,\ldots, w_\nu, w_j)| (w_1,w_2,\ldots, w_\nu)\in W, 1\le j\le \nu\},$$
    then
     $\mathcal{V}^{(i)}=G(\mathcal{U}^{(i)})$ is
     open in $\C^{\nu+1}$
     and
$$G: \mathcal{U}^{(i)}\setminus G^{-1}(C^v(G))\to \mathcal{V}^{(i)}\setminus C^v(G)$$
is an unbranched covering map;
\item \label{item2}
for each $i=1,\cdots,m$, there is {\it no} path $\gamma:[0,1)\to \mathcal{U}^{(i)}$, $\gamma(t)=(w(t),z(t))$ as follows:
$w(t)=(w_1(t),\cdots,w_\nu(t))$ where $w_j(t)=w_{j'}(t)$ if and only if $v_j=v_{j'}$,
$G_{w(t)}(z(t))=w_j(t)$ and $z(t)\notin C(G_{w(t)})$ for some $j$ and all $t$, finally,
as $t\to 1$, $w(t)\to w_\ast\in W$ and $z(t)\to \partial U^{(i)}_{w_\ast}\cup C(G_{w_\ast})$.
%
\end{enumerate}
In particular, by the property (4) and (\ref{ccv}), we have
\begin{remark}\label{ccvdef}
If $p_r(\textbf{v})\in U_g$ for some $r$, then there is $j=1,\cdots,\nu$ such that
$G_w(p_r(w))=w_j$ for all $w\in W$.
\end{remark}

Note that this definition allows the domain and range of the maps $G_w$,
as well as the sets $C(G_w),C^v(G_w)$ to depend on  $w$. Condition (\ref{item2}) rules out the situation
that some preimage of a singular value which is not in $C(G_w)$ moves to the boundary or collides with a point of $C(G_w)$ along some curve in $W$.

\begin{remark}\label{rmk:fexample}
In particular, the conditions hold if $G_{w}$ is a local analytic family of finite ramified coverings
with constant multiplicities at the critical points. Other non-trivial families include
\begin{itemize}
\item In the set-up of Theorem~\ref{single}, $(g:U_g\to \C, \{0\}, c_1)\in \HH$ and the triple $(g, G_w, \textbf{p})_{U_g}$ is a local holomorphic deformation of $(g, \{0\}, c_1)$.
\item In the set-up of Theorem~\ref{thm:classE} with $f\in\mathcal{E}$, putting $C(g)=\{z\in D: g'(z)=0\}=\{p_j\}_{j=1}^N$, we have $(g|U, C(g), \{1\})\in \HH$. Moreover, putting $p_j(w)\equiv p_j$, the triple $(g, G_w, \{p_j\})_W$ is a local holomorphic deformation of $(g|U, C(g), \{1\})$.
\end{itemize}
\end{remark}

\subsection{Global deformations of maps in $\HH$ and the separation property}\label{subsec:globaldeform}
Let us fix a marked map $g_0$ as in Section 2.  Assume that $g_0: U\to \C$ extends to a holomorphic map $g:U_g\to V_g$ and there exists a discrete subset $C(g)$ of $\C$  such that $P_0\subset C(g)$ and
$(g, C(g), \textbf{v})\in\HH$, where $\textbf{v}=\textbf{c}_1(g_0)$.
(In particular, $P\setminus P_0\subset U_g$ and $P_0\subset \overline{U}_g$). 
In this subsection we will define the notion of a (global) holomorphic deformation in the current
more general setting.


Let
$\mathbb{S}=\{S_x\}_{x\in g(P)}$ be a collection of connected open subsets of $\C$.
Let
\begin{equation*}
S_\ast^\nu:= \left\{(w_1,\cdots,w_\nu)| w_j\in S_{c_{1,j}}
\text{ and } w_j=w_{j'} \text{ if and only if } c_{1,j}=c_{1,j'}\right\}.
\end{equation*}
For a complex manifold $\Lambda$ with $\ast\in \Lambda$,
a local holomorphic deformation $(g,G, \textbf{p})_W$ of $(g, C(g),\textbf{v})$ {\em determines a (global) holomorphic deformation} over $(\Lambda,\ast)$ {\em for $\mathbb{S}$}, if the following holds: for any holomorphic map
$\rho\colon \Lambda  \to S_\ast^\nu$, there exist holomorphic functions $p_{\rho, j}:\Lambda\to \C$, and for each $\lambda\in \Lambda$, there exists a holomorphic map $g_{\rho,\lambda}: U_{\rho,\lambda}\to \C$ such that
\begin{itemize}
\item $(g_{\rho, \lambda},\textbf{p}_{\rho}(\lambda), \rho(\lambda))\in\HH$, where  $\textbf{p}_{\rho}=\{p_{\rho, j}\}_{j=1}^N$
 for all $\lambda\in \Lambda$;
\item $g_{\rho,\lambda}=G_{\rho(\lambda)}$, $\textbf{p}_\rho(\lambda)=\textbf{p}_{\rho(\lambda)}$ when $\rho(\lambda)\in W$ (so for $\lambda$ close to $\ast$);
\item  for each $\lambda_0\in \Lambda$, there exists a local holomorphic deformation $(g',G',\textbf{p}')_{W'}$ of $(g_{\rho,\lambda_0},\textbf{p}_\rho(\lambda_0), \rho(\lambda_0))$  so that
$g_{\rho,\lambda}=G'_{\rho(\lambda)}$ and $\textbf{p}'(\rho(\lambda))=\textbf{p}_\rho(\lambda)$ for $\lambda$ close to $\lambda_0$.
\end{itemize}
\begin{remark}\label{monodromy}
In other words, such a global holomorphic deformation is determined through analytic continuation by the local deformations $(g',G',\textbf{p}')_{W'}$.
If $\Lambda$ is simply connected, then by the Monodromy Theorem,
it is enough to check $(g,G,\textbf{p})_W$ admits an analytic continuation by the local deformations along every arc in $S^\nu_\ast$.
\end{remark}


We say that $(g,G, \textbf{p})_W$ has the {\em separation property} {\em with respect to the collection $\mathbb{S}$} if it determines global holomorphic deformations
over $(\D,0)$ for $\mathbb{S}$ and
for any holomorphic $\rho\colon \D\to S_\ast^\nu$, the corresponding global family $g_{\rho,\lambda}\colon \cup_{i=1}^m U^{(i)}_{g_{\rho,\lambda}}\to \cup_{i=1}^m V^{(i)}_{g_{\rho,\lambda}}$ and $p_{\rho, j}$,
one has, for all $\lambda\in \D$

\begin{equation}\label{extincl}
U_{g_{\rho,\lambda}}^{(i)}\subset S_x \text{ and } S_{g(x)}\subset V_{g_{\rho,\lambda}}^{(i)}
\end{equation}
whenever $x\in  g(P)\cap U^{(i)}_g$ and, moreover,
\begin{equation}\label{extincl'}
p_{\rho, j}(\lambda)\in S_x
\end{equation}
if $x=p_{\rho,j}(0)\in P_0\cap g(P)$.
In particular, $x\in S_x$ for all $x\in g(P)$.
It has {\it the robust separation property} if there exist $R>0$, $\epsilon>0$  such that
for any such $\rho\colon \Lambda\to S_\ast^\nu$ and $x,i$ as in~(\ref{extincl}),
$V^{(i)}_{g_{\rho,\lambda}}\subset B(0,R)$ and $V_{g_{\rho,\lambda}}^{(i)}$ contains an $\epsilon$-neighbourhood of $S_{g(x)}$.

Examples of triples with such separation properties include (cf Remark~\ref{rmk:fexample}):
\begin{itemize}
\item In the set-up of Theorem~\ref{single}, choosing $S_x= U$ for all $x\in g(P)$, the triple $(g: U_g\to\C, G_w, \textbf{p})_{U_g}$ determines a global holomorphic deformation over $(\D,0)$ in a trivial way and satisfies the separation property with respect to $\{S_x\}_{x\in g(P)}$.
\item In the set-up of Theorem~\ref{thm:classE} with $f\in\mathcal{E}$, choosing $S_x=\C_*$ for $x\in g(P)$, the triple $(g, G_w, \{p_j\})_W$ determines a global deformation over $(\D,0)$ in a trivial way and satisfies  the separation property with respect to $\{S_x\}_{x\in g(P)}$.
\end{itemize}

\subsection{The separation property implies the weak lifting property}
\begin{theorem}\label{thm:holo}
Let $g_0$ be a marked map with a local holomorphic deformation
$(g_0,G_0, \textbf{p}_0)_W$ and $\textbf{v}=\textbf{c}_1(g_0)$. Suppose that there is a holomorphic extension $g: U_g\to \C$ of $g_0$
and a discrete subset $C(g)$ of $\C$ such that $(g, C(g),\textbf{v})\in\HH$ and $(g_0, G_0, \textbf{p})$ extends to a local holomorphic deformation $(g, G, \textbf{p})_W$ of $(g, C(g)), \textbf{v})$, i.e.
\begin{itemize}
\item $G_w(z)=G_0(w,z)$ for all $(w,z)\in W\times U$,
\item $\textbf{p}_0=(p_1,\cdots,p_\nu)$ is a subset of $\textbf{p}=(p_1,\cdots,p_\nu,\cdots)$.
\end{itemize}
Assume that $(g,G,\textbf{p})_W$ has the separation property with respect to some $\textbf{S}=\{S_x\}_{x\in g(P)}$. Then
\begin{itemize}
\item  the triple $(g_0,G_0, \textbf{p}_0)_W$ has the weak lifting property. 
\item If $\C\setminus S_x$ contains at least $2$ points for each $x\in g(P)$,  then $(g_0,G_0, \textbf{p}_0)_W$ has the lifting property, in particular, the alternative in the conclusion of Theorem~\ref{thm:1eigen} holds.
\item If, moreover, $(g,G, \textbf{p})_W$ has the robust separation property, then the spectral radius ofp the associated operator $\mathcal{A}$ is strictly less than $1$.
\end{itemize}
\end{theorem}

\begin{proof}[Proof of Theorem~\ref{thm:holo}]
It is enough to prove that the triple $(g_0,G_0, \textbf{p}_0)_W$ has the weak lifting property. Indeed, if $\#(\C\setminus S_x)\ge 2$ for each $x\in g(P)$, the lifting property then follows from Montel's theorem. And if additionally $(g,G, \textbf{p})_W$ has the robust separation property,
Lemma~\ref{lem:perturb2spectrum} applies.   

Let $h_\lambda$ be a holomorphic motion of $g(P)$ over $(\D_\eps, 0)$ such that $h_\lambda(x)\in S_x$ for all $x\in g(P)$ and $\lambda\in \D_\epsilon$ and let $\widehat{h}_\lambda$ be the lift of $h_\lambda$ over 
$(\D_{\eps'},0)$ for some $\eps'\in (0,\eps)$. We shall prove that $\widehat{h}_\lambda$ extends to a holomorphic motion of $g(P)$ over $(\D_\eps, 0)$ and $\widehat{h}_\lambda(x)\in S_x$ holds for all $x\in g(P)$ and $\lambda\in \D_\eps$. Once this is proved, the weak lifting property follows.
%

Let $g_{\rho,\lambda}: \bigcup_{i=1}^m U^{(i)}_{\rho,\lambda}\to \C$ and $\textbf{p}_{\rho}(\lambda)$,
be the families corresponding to $\rho(\lambda)=(h_\lambda(c_{1,1}),\cdots,h_\lambda(c_{1,\nu}))$.
By the definition of lifting, for $\lambda\in \D_{\eps'}$, the following equations hold:
\begin{itemize}
\item for each $j=1,2, \cdots, \nu$ and $c_{0,j}\in g(P)$,
\begin{equation}\label{p0}
\widehat{h}_\lambda(c_{0,j})=(\textbf{p}_\rho(\lambda))_j;
\end{equation}
\item for each $x\in g(P)\setminus P_0$, we have
\begin{equation}\label{p}
g_{\rho,\lambda}(\widehat{h}_{\lambda}(x))=h_{\lambda}(g(x)).
\end{equation}
\end{itemize}
Let us first show that
\begin{enumerate}
\item [(i)] $\widehat{h}_\lambda(x)$ extends to a holomorphic function in $\D_\epsilon$ for each $x\in g(P)$;
\item [(ii)] for $x\in C(g)\cap g(P)$, $\widehat{h}_\lambda(x)\in C(g_{\rho,\lambda})$ and $\widehat{h}_\lambda(x)=(\textbf{p}_\rho(\lambda))_r$ for some $r$;
\item [(iii)] for $x\in g(P)\setminus C(g)$, $\widehat{h}_\lambda(x)\not\in C(g_{\rho, \lambda})$.
\end{enumerate}

{\em Case 1.} If $x=c_{0,j}\in g(P)$ for some $j$, the continuation of $\widehat{h}_\lambda(x)$ to $\D_\eps$ is defined by (\ref{p0}). Moreover, $(\textbf{p}(\lambda))_j\in S_x$ by (\ref{extincl'}).

{\em Case 2.} $x\in g(P)\setminus P_0$ and $g(x)\not\in \{c_{1,1}, \ldots, c_{1,\nu}\}$. Find $i\in \{1,\cdots,m\}$ so that $x\in U_g^{(i)}$.
Let $\mathcal{U}=\{(\lambda, z): \lambda\in \D_\eps, z\in U^{(i)}_{\rho,\lambda}\}$, $\mathcal{V}=\{(\lambda, z): \lambda\in \D_\eps, z\in g_{\rho,\lambda}(U^{(i)}_{\rho,\lambda})\}$,
$C^v(G)=\{(\lambda, z):z=h_\lambda(c_{1,j}), j=1,2,\ldots, \nu\}$. Then $G: \mathcal{U}\setminus G^{-1}(C^v(G))\to \mathcal{V}\setminus C^v(G)$ is an unbranched covering. Since $(\lambda, h_{\lambda}(g(x)))\in \mathcal{V}\setminus C^v(G)$ and since $\D_\eps$ is simply connected, it follows that $\widehat{h}_\lambda(x)$ extends to a holomorphic map in $\D_\eps$.


{\em Case 3.} $x\in g(P)\cap C(g)\setminus P_0$. So $x=p_r(\textbf{v})\cap U_g$ for some $r$. By Remark~\ref{ccvdef}, there is $j\in \{1,2,\ldots, \nu\}$ such that
$g_{\rho,\lambda} ((\textbf{p}_\rho(\lambda))_r)=h_\lambda(c_{1,j})$ holds for all $\lambda\in \D_\eps$. Since $g'(x)\not=0$, it follows that $\widehat{h}_\lambda(x)=(\textbf{p}_\rho(\lambda))_r$  for $\lambda\in \D_{\eps'}$. Thus $\widehat{h}_\lambda(x)$ extends to a holomorphic function in $\D_\eps$.

{\em Case 4.} $x\in g(P)\setminus (P_0\cup C(g))$ and $g(x)=c_{1,j}$ for some $j$.
Let us show that $\widehat{h}_\lambda(x)$ admits an analytic continuation along every curve $\Gamma: [0,1]\to \D_\eps$, $\Gamma(0)=0$. Assume the contrary and let $t_\ast\in (0,1)$ be the maximal so that $\widehat{h}_\lambda(x)$ has an analytic continuation $H: [0,t]\to \C$ for all $0\le t<t_\ast$.
By~(\ref{p}), $g_{\rho,\Gamma(t)}(H(t))=h_{\Gamma(t)}(g(x))$ for all $0\le t<t_\ast$. In particular, $H(t)\in U^{(i)}_{\rho,\Gamma(t)}$ for
$0\le t<t_\ast$.
Then $H(t)\to \partial U_{\rho, \lambda_*}\cup C(g_{\rho,\lambda_*})$ as $t\to t_\ast$. However, this is ruled out by property (6) in the definition of local holomorphic deformation of in class $\HH$. By the Monodromy theorem, $\widehat{h}_\lambda(x)$ extends to a holomorphic function in $\D_\eps$. By property (6) of local holomorphic deformation of in class $\HH$ again, $\widehat{h}_\lambda(x)\not \in  C(g_{\rho, \lambda})$ for all $\lambda$.

Since all $\widehat{h}_\lambda$ are holomorphic in $\D_\eps$, the equations (\ref{p0})-(\ref{p}) hold for all $\lambda\in \D_\eps$. Therefore, $\widehat{h}_\lambda(x)\in S_x$ holds for all $\lambda$.

Finally let us prove that $\widehat{h}_\lambda(x_1)\not=\widehat{h}_\lambda(x_2)$ holds for all $\lambda\in \D_\eps$ and $x_1\not=x_2$ in $g(P)$. Assume the contrary. By property (1) in the definition of local holomorphic deformation for triples in $\HH$ and property (ii) and (iii) above, we must have $x_1, x_2\in g(P)\setminus C(g)$. Let $\lambda_*$ be an accumulation point of the set $\Omega:=\{\lambda: \widehat{h}_\lambda(x_1)\not=\widehat{h}_\lambda(x_2)\}$ in $\D_\eps$. Then $\widehat{h}_{\lambda_*}(x_1)=\widehat{h}_{\lambda_*}(x_2)$ which implies that
$h_{\lambda_*}(g(x_1))=h_{\lambda_*}(g(x_2))$, hence $g(x_1)=g(x_2)$. Take a sequence $\lambda_n\in \Omega$ such that $\lambda_n\to\lambda_*$. Then $g_{\rho,\lambda_n}(z)= h_{\lambda_n}(g(x_1))$ has two distinct solutions. It follows that $\widehat{h}_{\lambda_*}(x_1)$ is a critical point of $g_{\rho,\lambda_*}$, hence $\widehat{h}_{\lambda_*}(x_1)=\widehat{h}_{\lambda_*}(x_2)\in C(g_{\rho,\lambda_*})$. However, by the property (iii) above, this implies that $x_1, x_2\in C(g)$, a contradiction.
\end{proof}

\section{The Arnol'd family}\label{sec:sinearnold}
In this section we will apply the above methods to a family of generalized Arnol'd  maps.
Let us fix an integer $d>0$. The (generalized) Arnol'd circle map is a map $A_{a,b}: \R/\Z\to \R/\Z$ of the form $A_{a,b}(t)=dt+a+b\sin(2\pi t)(\!\!\!\mod 1)$.
(Choice $d=1$ corresponds to the standard Arnol'd map.) Here $(a,b)\in \R\times \R\setminus \{0\}$ (in fact, $a\in \R/\Z$). We consider pairs $(a,b)$ such that $A_{a,b}$ has two (distinct) real critical values $v_j=A_{a,b}(e_j)$,
where $e_j$, $j=1,2$ are two (distinct) real critical points of $A_{a,b}$. Then one checks that given $(a_0,b_0)$ the correspondence $\textbf{q}_A: (a,b)\mapsto (v_1,v_2)$ is a local homeomorphism which extends to a biholomorphic map from a (complex) neighbourhood $X\subset \C\times \C$ of $(a_0,b_0)$ onto a (complex) neighbourhood $Y$ of $\textbf{q}_A((a_0,b_0))$ and there is also another biholomorphic map
$\textbf{p}_A=(p_1^A,p_2^A): (v_1,v_2)\mapsto (e_1,e_2)$ from $Y$ onto $\textbf{p}_A(Y)$.
Let $F_{(v_1,v_2)}=A_{\textbf{q}_A^{-1}((v_1,v_2))}$ for $(v_1,v_2)\in Y$.
\begin{theorem}\label{arn}
Suppose that for some real parameters $(a_0,b_0)$ the map $A_{a_0,b_0}: \R/\Z\to \R/\Z$ has two different critical values $c_{1,1}$, $c_{1,2}$ and each of them is either periodic or preperiodic. Consider the real marked map $A_{a_0,b_0}$ and
the real local holomorphic deformation $(A_{a_0,b_0}, F, \textbf{p}_A)_Y$ of $A_{a_0,b_0}$.
Then the 'positively oriented' transversality property (\ref{eq:trans2}) holds:
$$\frac{\det (D\mathcal{R}(\textbf{c}_1)))}{\prod_{j=1}^\nu Dg^{q_j-1}(c_{1,j})}>0.$$

\end{theorem}
\begin{proof}
The projection $E(z)=\exp(2\pi i z)$ semi-conjugates $A_{a,b}$ to the following map $f_{\mu,b}$:
$$f_{\mu,b}(z)=\mu z^d \exp\{\pi b (z-1/z)\}.$$
That is,
\begin{equation}\label{econj}
E\circ A_{a,b}=f_{\mu,b}\circ E,
\end{equation}
where $\mu=E(a_0)$.
We consider pairs $(\mu, b)$ such that
$f_{\mu,b}$ has two distinct critical values $w_1$ and $w_2$. Then the correspondence $\textbf{q}: (\mu, b)\mapsto (w_1, w_2)$ is a local biholomorphic map. Let $\textbf{p}=(p_1,p_2)$ be another a local biholomorphic map where $\tilde e_1=p_1(w_1,w_2)$, $\tilde e_2=p_2(w_1,w_2)$ are the two critical points of $G_{\textbf{w}}=f_{\mu,b}$ with $\textbf{w}=(w_1,w_2)=\textbf{q}((\mu,b))$.

By~(\ref{econj}), the map $f_{\mu_0,b_0}$ where $\mu_0=E(a_0)$ is a marked map and by the above $(f_{\mu_0,b_0}, G, \textbf{p})$ is its local holomorphic deformation. Note that
$E\circ F_{v_1,v_2}=G_{w_1,w_2}\circ E$ where $w_i=E(v_i)$ and $p_i^A=p_i\circ E$, $i=1,2$. By this, it is easy to check that $\mathcal{A}_G I=I \mathcal{A}_F$ where $\mathcal{A}_G$, $\mathcal{A}_F$ are the transfer operators to the holomorphic deformations $(f_{\mu_0,b_0}, G, \textbf{p})$ and  $(A_{a_0,b_0}, F, \textbf{p}_A)_Y$ respectively and $I$ is a non-degenerate diagonal matrix. Hence, $\mathcal{A}_G$, $\mathcal{A}_F$ share the same spectrum. Therefore, by Corollary~\ref{real}, it is enough to show that: (i) every cycle of $A_{a_0,b_0}$ is either superattarcting or repelling and (ii) the spectrum of $\mathcal{A}_G$ belongs to $\{|t|\le 1\}\setminus \{1\}$.

(i) Holomorphic function $f_{\mu_0, b_0}: \C_\ast\to \C_\ast$ where $C_\ast=\C\setminus \{0\}$ has the only asymptotic values $0$ and $\infty$. It is also critically finite,  hence
every cycle of $f_{\mu_0, b_0}$ is either superattracting or repelling. By~(\ref{econj}), this holds for $A_{a_0,b_0}$ as well.

To show (ii), it is enough to prove that:

(a) the triple $(f_{\mu_0, b_0},G,\textbf{p})_W$ has the lifting property so that Theorem~\ref{thm:1eigen} applies,

(b) the second part of the alternative in Theorem~\ref{thm:1eigen} does not hold.

{\it Proof of (a)}. Each map $f_{\mu,b}: \C_\ast\to \C_\ast$ is a holomorphic covering map in $\HH$ and $(g,G,\textbf{p})_W$ is
a local holomorphic deformation of $g=f_{\hat\mu,\hat b}$ (in the sense of Subsection~\ref{subsec:localdefn})
for any pair
$(\hat\mu,\hat b)\in \C_*\times \C_*$ such that $g$ has two distinct critical values. Here $W$ is a small enough neighbourhood of $\textbf{q}^{-1}((\hat\mu,\hat b))$. Indeed, the properties (1)-(5) are straightforward.
To check the property (6), let us assume the contrary, i.e., $\mu(t)z(t)^d\exp{\pi b(t)(z(t)-1/z(t))}=w_j(t)$ where $\mu,z,b,w_j:[0,1)\to \C$ are continuous, $w_j(t)$ is a coordinate of $\textbf{q}^{-1}((\mu(t),b(t)))$,
$z(t)\notin C(f_{\nu(t),b(t)})$ for all $t\in [0,1)$ while $(\mu(t),b(t))\to (\nu_*,b_*)\in \C_*\times \C_*$ and $z(t)\to \partial\C_*\cup C(f_{\mu_*,b_*})$ as $t\to 1$. The case $z(t)\to C(f_{\mu_*,b_*})$ is impossible as
otherwise the critical point $z(1)=\lim_{t\to 1} z(t)$ of $f_{\mu_*,b_*}$ would be multiple. So, let $z(t)\to \{0,\infty\}$. Consider $z(t)\to \infty$ (the other case is similar).
Then $z(t)^d\exp\{\pi b(t)z(t)\}\to \lim_{t\to 1} w_j(t)/\mu_*\not=0,\infty$ as $t\to 1$. Hence, $u(t)^d\exp{u(t)}\to B\not= 0,\infty$ as $u(t)=\pi b(t)z(t)$ is continuous on $[0,1)$ and $u(t)\to \infty$ as $t\to 1$.
This is a contradiction since the asymptotic values of the function $z^d\exp(z)$ are $0$ and $\infty$.

Take $\mathbb{S}=\{S_x\}_{x\in P}$ where $S_x=C_*$ for all $x\in P$. It follows from the Monodromy theorem that
the local holomorphic deformation $(f_{\mu_0,b_0},G, \textbf{p})_W$  where $W$ is a small neighbourhood of $\textbf{q}((\mu_0,b_0))$ determines  a global holomorphic deformation over $(\Delta,0)$ for $\mathbb{S}$ in the sense of  Subsection~\ref{subsec:globaldeform}.
Moreover, $(f_{\mu_0,b_0},G, \textbf{p})_W$ has the separation property with the choice $\mathbb{S}$.

By Theorem~\ref{thm:holo}, $(f_{\mu_0, b_0},G,\textbf{p})$ has the weak lifting property: given a holomorphic motion
$h_\lambda^{(0)}$ of $P$ over $(\Lambda, 0)$, there exists $\eps_0=\eps>0$ and holomorphic motions
$h_\lambda^{(k)}$, $k=1,2,,\cdots$,
of $P$ over $(\Delta_{\eps}, 0)$ such that
for each $k=0,1,\cdots$, there is $\epsilon_{k+1}>0$ such that $h_\lambda^{(k+1)}$ is the lift of $h_\lambda^{(k)}$ over $(\Delta_{\eps_{k+1}}, 0)$ for $(f_{\mu_0,b_0}, G, \textbf{p})$,
We have to show that the family $h^{(k)}_\lambda$, $k=0,1,...$, is uniformly  bounded in a neighbourhood of $\lambda=0$.
  We always have that $h^{(k)}_\lambda(x)\not=0$ for all $x\in g(P)$ and all $|\lambda|<\eps$. For $k\ge 0$, let $\tilde e_1^{(k)}(\lambda)$, $\tilde e_2^{(k)}(\lambda)$ be two critical points of the map $G_{\textbf{c}_1(\lambda)}$ where $\textbf{c}_1(\lambda)=(h^{(k)}_\lambda(E(c_{1,1})), h^{(k)}_\lambda(E(c_{1,2})))$.
Note that $\tilde e^{(k)}_j(0)=\tilde e^{(0)}_j(0)$ for $j=1,2$ and all $k$. The family $(\tilde e^{(k)}_1(\lambda)/\tilde e^{(k)}_2(\lambda))_{k\ge 0}$ is normal in $\Delta_\eps$ as it does not take values $0,1,\infty$.
Observe that
$\tilde e^{(k)}_1(\lambda) \tilde e^{(k)}_2(\lambda)=1$. Hence, each family $(\tilde e^{(k)}_j(\lambda))_{k\ge 0}$, $j=1,2$ is normal in $\Delta_\eps$. In particular, this proves the claim (a) if $P=P_0$.
If $P\setminus P_0\not=\emptyset$, i.e., say, $c_{0,1}$ is not a fixed point of $F_{a_0, b_0}$, then for every $x\in P\setminus P_0$,
functions $h^{(k)}_\lambda(x)/\tilde e_{1}^{(k-1)}(\lambda)$ are holomorphic in $\Delta_\eps$ and do not take values $0$ and $1$ and, hence, form a normal family in $\Delta_\eps$.
Therefore, $(h^{(k)}_\lambda(x))_{k\ge 0}$ is normal in $\Delta_\eps$ too. As at $\lambda=0$ it is bounded, it is bounded
in $\Delta_{\eps/2}$.
This proves (a).

{\it Proof of (b)}. Assume the contrary. The
there exists a neighbourhood $W$ of the point ${\bf c}_1=(c_{1,1},c_{1,2})$
such that the equation
\begin{equation}\label{allornoth2}
\mathcal{R}(\textbf{w})=0
  \end{equation}
defines an analytic variety $E$ in $W$ of (complex) dimension at least $1$. Note that for all such $\textbf{w}=(w_1,w_2)$ the map $G_{\textbf{w}}$ is critically finite.
and the repelling periodic points are dense in the Julia set.
It follows that there is a pair of non-trivial holomorphic maps $\mu(t)$, $b(t)$, $|t|<\delta$, such that
$\mu(0)=\mu_0$, $b(0)=b_0$ and the points in the Julia $J_t$ set of $f_{\mu(t),b(t)}$ move holomorphically in $t$.
For all $t$, the complement $\C_\ast\setminus J_t$ is either empty of consists of basins of attraction of superattracting cycles.
By the $\lambda$-Lemma, $f_{\mu(t),b(t)}$ is quasi-conformally conjugate to $f_{\mu_0,b_0}$. Moreover, if $J_0$ has zero area, the conjugacy can be chosen to be conformal, i.e. a Mobius transformation $M_t$.
If $J_0$ has positive area,
$f_{\mu_0,b_0}$ has an invariant line field on its Julia set $J_0$. Considerations which are similar to the proof of Theorem 3.17, \cite{McM} show that it must be holomorphic. It follows that in this case as well
$f_{\mu(t),b(t)}$ is conjugate to $f_{\mu_0,b_0}$ by a Mobius transformation $M_t$. As $M(\{0,\infty\})=\{0,\infty\}$, $M(z)=c/z$ for some $c\not=0$. Then necessarily $c=1$ and $\nu(\lambda)=\nu_0$ for all $\lambda$
where $\mu_0=\pm 1$.
Hence, for every $t$ the map $f_{\mu_0,b(t)}$ is critically finite, which is possible only if $b(t)$ is a constant function, too, a contradiction.
\end{proof}

\part*{Part C: The method applied to the family  $f_c(x)=|x|^{\ell_\pm}+c$}

\section{The family $f_c(x)=|x|^{\ell_\pm}+c$ with $\ell_{\pm}>1$ large}\label{sec:finiteorder}
\subsection{Unimodal family}
In the next theorem we obtain monotonicity for unimodal (not necessary symmetric!) maps
in the presence of critical points of large non-integer order, but only
if not too many points in the critical orbit are in the orientation reversing branch.

\begin{theorem}\label{thm:finiteorder2}
Fix real numbers $\ell_-,\ell_+\ge 1$ and consider the family of unimodal maps $f_c=f_{c,\ell_-,\ell_+}$ where
$$f_c(x)=\left\{\begin{array}{ll}
|x|^{\ell_-}+c & \mbox{ if } x\le 0\\
|x|^{\ell_+}+c & \mbox{ if } x\ge 0.
\end{array}
\right.
$$
For any integer $L\ge 1$ there exists $\ell_0>1$
so that for any $q\ge 1$ and any
periodic kneading sequence   $\bold i=i_1i_2\cdots\in \{-1,0,1\}^{\Z^+}$ of period $q$
so that
$$\#\{0\le j< q ; i_j =-1 \}\le L,$$
and any pair $\ell_-,\ell_+\ge \ell_0$ there is at most one $c\in\R$ for which the kneading sequence
of $f_c$ is equal to $\bold i$. Moreover,
\begin{equation}\label{fiortrans2}
\sum_{n=0}^{q-1} \frac{1}{Df_c^n(c)}>0.
\end{equation}
\end{theorem}

{\bf Notations.} As usual, for any three distinct point $o, a, b\in\C$, let $\angle aob$ denote the angle in $[0,\pi]$ which is formed by the rays $oa$ and $ob$. We shall often use the following obvious relation: for any distinct four points $o,a, b, c$,
$$\angle aob +\angle boc \ge \angle aoc.$$
For $\theta\in (0, \pi)$, let
$$D_{\theta}=\{z\in\C\setminus\{0,1\}: \angle 0z1>\pi-\theta\}$$ and let
$$S_\theta=\{re^{it}: t\in (-\theta, \theta)\}.$$ For $0<t<1$, we shall only consider $z^t$ in the case $z\not\in (-\infty, 0)$ and $z^t$ is understood as the holomorphic branch with $1^t=1$.

Let us fix a map $f=f_{c,\ell_-,\ell_+}$ with a periodic critical point of period $q$ and let $P=\{f^n(0): n\ge 0\}$. So $P$ is a forward invaraint finite set.
Denote
$$\ell=\min\{\ell_-,\ell_+\}.$$
\begin{definition}
A holomorphic motion $h_\lambda$ of $P$ over $(\Omega,0)$, is called ${\theta}$-regular if
\begin{enumerate}
\item[(A1).] For $a\in P$, $$h_{\lambda}(a)\in S_{4\theta/\ell} \text{ , if } a>0$$
and
$$h_{\lambda}(a)\in -S_{4\theta/\ell} \text{ , if } a<0;$$
\item[(A2).] For $a, b\in P$, $|a|>|b|>0$ and $ab>0$, $$\frac{h_{\lambda}(b)}{h_{\lambda}(a)}\in D_{\theta}.$$
\end{enumerate}
\end{definition}

Given a $\theta$-regular holomorphic motion $h_{\lambda}$ of $P$ over $\Omega$, with $\theta\in (0, \pi)$, one can define another holomorphic motion $\tilde{h}_{\lambda}$ of $P$ over the same domain $\Omega$ as follows: $\tilde{h}_{\lambda}(0)=0$; for $a\in P$ with $a>0$,
$$\tilde{h}_\lambda(a)=(h_{\lambda}(f(a))-h_{\lambda}(f(0)))^{1/\ell_+};$$ for $a\in P$ with $a<0$, define
$$\tilde{h}_\lambda(a)=-(h_{\lambda}(f(a))-h_{\lambda}(f(0)))^{1/\ell_-}.$$

 The new holomorphic motion is called the {\em lift} of $h_{\lambda}$ which clearly satisfies the condition (A1), but not necessarily (A2) in general.

\begin{mainlemma} There is $\ell_0$ depending only on the number $L$  such that for any $\ell\ge \ell_0$ and each $\theta$ small enough, the following holds:
If $\#\{0\le j< q ; i_j =-1 \}\le L$ and
if  a $\theta$-regular motion can be successively lifted $q-1$ times and all these successive lifts are $\theta$-regular, then the $q$-th lift of the holomorphic motion is $\theta/2$-regular.
\end{mainlemma}

\begin{proof}[Proof of Theorem~\ref{thm:finiteorder2}]
Given $L$, choose $\ell_0$ as in the Main Lemma. It is enough to prove (\ref{fiortrans2}) provided $\ell\ge \ell_0$. Consider a local holomorphic deformation $(f_c,f_w,\textbf{p})_W$ where $W\subset \C$ is a small neighbourhood of $c$, $f_w=f_c+(w-c)$ and $\textbf{p}=0$.
Let $h_\lambda$ be a holomorphic motion of $P$ over $(\Delta,0)$. Let us fix $\theta>0$ small enough.
Restricting $h_\lambda$ to a smaller domain $\Delta_\eps$, we may assume that $h_{\lambda}$ is $\theta$-regular and that $h_\lambda$ can be lifted successively for $q$ times. Therefore by the Main Lemma, we obtain a sequence of holomorphic motions $h^n_\lambda$ of $P$ over $(\Delta_\eps,0)$, such that $h^0_\lambda=h_{\lambda}$ and $h^{n+1}_\lambda$ is the lift of $h^n_\lambda$
and such that $h^n_\lambda(x)\in \pm S_{\theta_n}$ for all $n$ and all $x\in P$ where $\theta_n\to 0$ as $n\to \infty$.
Thus $(f_c,f_w,\textbf{p})_W$ has the lifting property and by Theorem~\ref{thm:1eigen}, the transversality condition (\ref{fiortrans2}) holds.

Alternatively, the uniqueness of $c$ follows directly from the Main Lemma. Indeed, let
$\tilde{f}=f_{\tilde c}$ be a map with the same kneading sequence as $f_{c}$. Then one can define a real holomorphic motion $h_{\lambda}$ over some domain $\Omega\ni 0,1$ such that $h_{\lambda}(f^n(0))=\tilde{f}^n(0)$ for $\lambda=1$. As above, for $i>0$ let $h_\lambda^i$ be the lift
of $h_\lambda^{i-1}$.  As we have just shown,
$h_\lambda^i(c)$ is contained in the sector $-S_{\theta_n}$ with $\theta_n\to 0$,
this sequence of functions $\lambda\to h_\lambda^i(c)$ has to converge to a constant function. Since by construction
of the lifts $\tilde{c}=h_1^n(c)$ for each $n\ge 1$  we conclude that $\tilde{c}=c$.
\end{proof}
\subsection{Proof of the Main Lemma}
\begin{lemma}\label{lem:Schwarz}
For any $\theta\in (0, \pi)$ and $0<t<1$, if $z\in D_{\theta}$ then $z^t\in D_{\theta}$.
\end{lemma}
\begin{proof} This is a well-known consequence of the Schwarz lemma, due to Sullivan.
\end{proof}

When $\angle 01z$ is much smaller than $\angle 10z$, we have the following improved estimate.
\begin{lemma} \label{lem:pbtriangle}
For any $\eps>0$, there is $\delta>0$  such that the following holds.
For $z\in D_{\theta}$ with $\theta\in (0, \pi/2]$ and $\angle 01z<\delta\theta$ and for any $0<t<1$, we have $\angle 01z^t< \eps \theta.$
\end{lemma}
\begin{proof} Write $z=re^{i\alpha}$ where $r>0$ and $\alpha\in (0, \theta)$ and write $\alpha'=t\alpha$ and $\beta'=\angle 01z^t$. By assumption, $\alpha+\beta\le \theta$. By the sine theorem,
$$r=\frac{\sin \beta}{\sin (\alpha+\beta)}$$
and $$r^t=\frac{\sin \beta'}{\sin (t\alpha +\beta')}.$$

If $\alpha+\beta<\eps\theta$ then by Lemma~\ref{lem:Schwarz}, $\alpha'+\beta'\le \alpha+\beta<\eps \theta.$ Assume now $\alpha+\beta\ge \eps \theta$.
Let $K>0$ be a large constant such that
$$\frac{t}{K^t-1}<\eps \text{ for any } 0<t<1.$$
Assume $\beta<\delta\theta$ for $\delta$ small. Then $r<1/K$. Thus
$$\tan\beta'=\frac{r^t \sin t\alpha}{1-r^t\cos t \alpha} \le \frac{tr^t}{1-r^t} \alpha\le \frac{t}{K^t-1}\alpha<\eps \theta.$$
\end{proof}

\begin{lemma} \label{lem:pbangle}
Let $\varphi_\lambda$ be a $\theta$-regular motion with $\theta\in (0, \pi/10]$ and let $\psi_\lambda$ be its lift. For $x,y\in P$ so that $xy\ge 0$
let $x_\lambda=\psi_\lambda(x)$, $y_\lambda=\psi_\lambda(y)$, $u_\lambda=\varphi_\lambda(f(x))$, $v_\lambda=\varphi_\lambda(f(y))$ and $c_\lambda=\varphi_\lambda(f(0))$.

For any $\eps>0$ there is $\ell_0$ and $\delta>0$ such that if $\ell>\ell_0$ then the following hold.
\begin{enumerate}
\item If $f(x)\le 0\le f(y)$ then $\angle 0x_\lambda y_\lambda \ge \pi-\eps \theta$ for all $\lambda$.
\item Let $0<f(x)<f(y)$. Then (i) $\angle 0x_\lambda y_\lambda \ge \angle 0 u_\lambda v_\lambda -\frac{8\theta}{\ell}$. If, moreover,
$c_\lambda\in -S_{\theta_1}$ and $u_\lambda,v_\lambda\in S_{\theta_1}$ for some $\theta_1\in (0, 4\theta/\ell]$ then (ii) $x_\lambda,y_\lambda\in \pm S_{\theta_1/\ell}$ and $\angle 0x_\lambda y_\lambda \ge \angle 0 u_\lambda v_\lambda -2 \theta_1$.
\item Suppose $f(x)<f(y)<0$ and
$$\alpha=\pi-\min (\angle c_\lambda v_\lambda 0, \angle u_\lambda v_\lambda 0)<\delta \theta.$$
Then $$\angle 0 xy \ge \pi-\eps \theta.$$
\end{enumerate}
\end{lemma}

\begin{proof} Note that $\triangle 0 xy$ is the image of $\triangle c_\lambda u_\lambda v_\lambda$ under an appropriate branch of $z\mapsto (z-c_\lambda)^t$.  Since $\angle xoy<8\theta/\ell$, an upper bound on $\angle oyx$ implies a lower bound on $\angle oxy$.

(1) In this case, we have $u_\lambda \in -\overline{S_{4\theta/\ell}}$ and $v_\lambda \in \overline{S_{4\theta/\ell}}$, so
$$\angle 0u_\lambda v_\lambda \le  4\theta/\ell,$$
and $$\angle 0v_\lambda u_\lambda \le 4\theta/\ell.$$
In particular,
$$\angle c_\lambda u_\lambda v_\lambda \ge \angle c_\lambda u_\lambda 0 - \angle 0u_\lambda v_\lambda \ge \pi-\theta - 4\theta/\ell\ge \pi-5 \theta.$$
By Lemma~\ref{lem:pbtriangle}, the statement follows.

(2) In this case,
$$\angle c_\lambda u_\lambda v_\lambda \ge \angle 0 u_\lambda v_\lambda -\angle 0 u_\lambda c_\lambda\ge  \angle 0 u_\lambda v_\lambda - 8\theta /\ell.$$
Thus by Lemma~\ref{lem:Schwarz}, the conclusion (i) follows; (ii) is similar.

(3) In this case,
$$\angle c_\lambda u_\lambda v_\lambda \ge \angle c_\lambda u_\lambda 0-\angle v_\lambda u_\lambda 0\ge \pi-\theta-\alpha\ge \pi-2\theta$$
and $$\angle c_\lambda v_\lambda u_\lambda\le 2\pi-(\angle c_\lambda v_\lambda 0+ \angle 0v_\lambda u_\lambda) \le 2\alpha.$$
So the conclusion follows from Lemma~\ref{lem:pbtriangle}.
\end{proof}

Now suppose that we have a sequence of $\theta$-regular holomorphic motions $h_\lambda^i$ of $P$, $i=0,1,\ldots, q-1$  over the same marked domain $(\Omega,0)$, such that $h_\lambda^{i}$ is a lift of $h_\lambda^{i-1}$ for all $1\le i<q$. Then $h^q_\lambda$, lift of $h_\lambda^{q-1}$ is well-defined and satisfies the condition (A1) with the same constant $\theta$.
For each $0\le i\le q$, $\lambda\in\Omega$ and $x,y\in P$ so that $0<|x|<|y|$ and $xy>0$, let
\begin{align*}
& \theta_\lambda^i(x,y)=\pi-\\
&\inf\{\angle 0h_\lambda^i(z_1) h_\lambda^i(z_2): z_1, z_2\in P, 0<|z_1|\le |x|<|y|\le |z_2|, \ \ xz_1>0, xz_2>0\}\\
\ge & \pi-\angle 0h^i_\lambda(x)h_\lambda^i(y).
\end{align*}
Furthermore,
given  any $x,y\in P$, $xy>0$ (but not necessarily $|x|<|y|$), denote
$$\hat{\theta}^i(x,y)=\theta^i(x\wedge y, x\vee y)$$ where $x\wedge y=x/|x|\min(|x|,|y|)$ and $x\vee y=x/|x|\max(|x|,|y|)$.

\begin{lemma}\label{lem:pbangle1}
Consider $0\le i<q$, $x,y\in P$ where $xy>0$ and $\lambda\in\Omega$. For any $\eps>0$ there is $\delta>0$ and $\ell_0>0$ such that if $\ell\ge \ell_0$, then the following hold.
\begin{enumerate}
\item If $f(x)\le 0\le f(y)$ then $$\hat{\theta}^{i+1}_\lambda(x,y)\le \eps\theta.$$
\item Let $r\ge 1$ be such that $i+r\le q$. If $0<f^j(x)<f^j(y)$ for all $1\le j\le r$, then $$\hat{\theta}_\lambda^{i+r}(x,y)\le \hat{\theta}_\lambda^i(f^r(x),f^r(y))+\eps\theta.$$
\item If $f(x)<f(y)<0$ and $\hat{\theta}_\lambda^i(f(x), f(y))<\delta \theta$,
then $$\hat{\theta}_\lambda^{i+1}(x,y)\le 4\max (\eps\theta, \hat{\theta}_\lambda^i(f(x), f(y))).$$
\end{enumerate}
\end{lemma}
\begin{proof}
Note that $f(x)<f(y)$ implies $|x|<|y|$.

(1) For each $0<|z_1|\le |x|<|y|\le |z_2|$ as in the definition of $\theta_\lambda^i(x,y)$ we have $f(z_1)\le 0$ and $f(z_2)\ge 0$.  So by Lemma~\ref{lem:pbangle} (1), (applying to $\varphi=h^i$ and $\psi=h^{i+1}$), $\angle 0h_\lambda^{i+1}(z_1)h_\lambda^{i+1}(z_2)\ge \pi-\eps\theta$.  Thus the statement holds.

(2) Consider $0<|z_1|\le |x|<|y|\le |z_2|$ so that $z_1z_2>0$. Then $f(z_2)>0$. If $f(z_1)\le 0$, then by Lemma~\ref{lem:pbangle} (1),
$\angle 0h_\lambda^{i+r}(z_1)h_\lambda^{i+r}(z_2)\ge \pi-\eps\theta$.
Assume $f(z_1)>0$ and let $r_1\in \{1,\cdots,r\}$ be maximal such that $0<f^j(z_1)\le f^j(x)<f^j(y)\le f^j(z_2)$ for all $1\le j\le r_1$.
Notice that then
$$0<f^{r_1}(z_2)\le f^{r_1-1}(f^2(0))<f^{r_1-2}(f^2(0))<\cdots<f^2(0).$$
Let us show that for all $k\in \{0,\cdots,r_1-1\}$,
\begin{equation}\label{r1}
h_\lambda^{k+i+r-r_1}(f(0))\in -S_{4\theta/\ell^{k+1}},
\end{equation}
and
\begin{equation}\label{r2}
h_\lambda^{k+i+r-r_1}(f^{r_1-k}(z_1)), h_\lambda^{k+i+r-r_1}(f^{r_1-k}(z_2))\in S_{4\theta/\ell^{k+1}}.
\end{equation}
Indeed, this holds for $k=0$ as $h_\lambda^{i+r-r_1}$ is $\theta$-regular. Now, for $1\le k\le r_1-1$, (\ref{r1})-(\ref{r2}) follows by a successive application of the second part of Lemma~\ref{lem:pbangle} (2).
This proves (\ref{r1})-(\ref{r2}).
In turn, using (\ref{r1})-(\ref{r2}) and again applying successively Lemma~\ref{lem:pbangle} (2),
\begin{align*}
&\angle 0h^{i+r}_\lambda(z_1) h^{i+r}_\lambda(z_2)> \angle 0h^{i+r-r_1}_\lambda(f^{r_1}(z_1)) h^{i+r-r_1}_\lambda(f^{r_1}(z_2))-2\sum_{k=0}^\infty \frac{4\theta}{\ell^{k+1}}=\\
&\angle 0h^{i+r-r_1}_\lambda(f^{r_1}(z_1)) h^{i+r-r_1}_\lambda(f^{r_1}(z_2))-\frac{8\theta}{\ell-1}.
\end{align*}
Consider two cases. If $r_1<r$, then $f^{r_1+1}(z_1)\le 0$ and $f^{r_1+1}(z_2)>0$ and by Lemma~\ref{lem:pbangle} (1),
$$\angle 0h^{i+r-r_1}_\lambda(f^{r_1}(z_1)) h^{i+r-r_1}_\lambda(f^{r_1}(z_2))\ge \pi-{\eps}\theta$$
for any $\ell$ large enough.
If $r_1=r$,
$$\angle 0h^{i}_\lambda(f^{r}(z_1)) h^{i}_\lambda(f^{r}(z_2))\ge \pi-\theta_\lambda^i(f^r(x),f^r(y)).$$
In any case,
$$\angle 0h^{i+r}_\lambda(z_1) h^{i+r}_\lambda(z_2)>\pi-\theta_\lambda^i(f^r(x),f^r(y))-\eps\theta$$
provided $\ell$ is large enough.
Thus the statement holds.

(3) Notice that in this case $\hat{\theta}^i_\lambda(f(x),f(y))=\theta^i_\lambda(f(y),f(x))$. Consider $0<|z_1|\le |x|<|y|\le |z_2|$ so that $z_1z_2>0$. If $f(z_2)>0$ then by Lemma~\ref{lem:pbangle} (1),
$\angle 0h_\lambda^{i+1}(z_1)h_\lambda^{i+1}(z_2)\ge \pi-\eps\theta$.
Assume $f(z_2)<0$. Then $0>f(z_2)\ge f(y)>f(x)>f(z_1)>c$. So
$$\angle h^i_\lambda(c)h^i_\lambda(f(z_2))0\ge \pi-\theta_\lambda^i(f(y), f(x))$$
and $$\angle h^i_\lambda(f(z_1)) h^i_\lambda (f(z_2))0 \ge \pi-\theta_\lambda^i (f(y), f(x)).$$
By Lemma~\ref{lem:pbangle} (3),
$$\angle 0 h^{i+1}_\lambda  (z_1)h^{i+1}_\lambda (z_2)\ge \pi- 4 \max (\theta_\lambda^i(f(y), f(x)), \eps\theta),$$
provided that $\theta^i_\lambda(f(y), f(x))/\theta$ is small enough and $\ell$ is large enough.
\end{proof}

\begin{proof}[Completion of proof of the Main Lemma]
It is easy to check that $h^q$ satisfies the condition (A1) with $S_{4\theta/\ell}$ replaced by $S_{2\theta/\ell}$. It remains to check that  for $x,y\in P$, $0<|x|<|y|$ and $xy>0$ implies $\angle 0 h_\lambda^q(x) h_{\lambda}^q(y)>\pi-\theta/2$. Since the critical point is periodic, there is a minimal integer $p$, less than the period $q$ of the critical point, such that $$f^p([x,y])\ni 0.$$
Let us define $p-1=m_0>m_1>\cdots>m_{j_0-1}>m_{j_0}=0$ inductively as follows. Given $m_i$, let $m_{j+1}\in \{0,1\cdots,m_j-1\}$ be the maximal so that $f^{m_{j+1}}([x,y])\subset \R^-$ if it exists and $m_{j+1}=0$ otherwise.
Note that $j_0\le L+1$.
Let
$$\kappa_{m_j} =\hat\theta_\lambda^{q-m_j} (f^{m_j}(x), f^{m_j}(y))/\theta, j=0,1,\ldots, j_0.$$
Fix $\eps>0$ small. Assume that $\ell$ is large.
Then by Lemma~\ref{lem:pbangle1} (1),
$$\kappa_{m_0}=\kappa_{p-1}\le \eps.$$
For each $0<j\le j_0$, by Lemma~\ref{lem:pbangle1} (2) and (3),
$$\kappa_{m_{j+1}}\le 4 \kappa_{m_j} + 4\eps$$
provided that $\kappa_{m_j}$ is small enough and $\ell$ is large enough.
Therefore, provided that $\ell$ is large enough, we have $\kappa_0<1/2$. It follows that
$$\angle 0 h^q_\lambda(x) h_\lambda^q(y) \ge \pi-\kappa_0\theta \le \pi-\theta/2.$$
\end{proof}

\section{The family $f_c(x)=|x|^\ell+c$ with $\ell$ odd}\label{sec:finiteoddorder}
In this section we will prove the following theorem.
\begin{theorem}\label{thm:finiteoddorder}
Let $\ell\ge 3$ be an odd integer. Suppose that $f_{c_1}(x)=|x|^\ell+c_1$ satisfies the following:
\begin{itemize}
\item there exists an integer $q\ge 1$ such that $f_{c_1}^{q+1}(0)=0$, $f_{c_i}^j(0)\not=0$ for $1\le j\le q$. In particular, $c_1<0$ and $f_{c_1}$ has an orientation reversing fixed point $-w<0$.
\item $f_{c_1}^j(0)\not\in [-w, 0)$ for any $1\le j\le q$.
\item $c_2>c_3>c_4>0$.
\end{itemize}
Then
$$\sum_{n=0}^q \frac{1}{Df_{c_1}^n(c_1)}>0.$$
\end{theorem}

In the proof it will be convenient to define, $z_1=|f_{c_1}^q(0)|$ so $f_{c_1}(z_1)=f_{c_1}(-z_1)=0$. Note that $|w|<|z_1|$.
Let $$P=\{f_{c_1}^j(0): 0\le j\le q\}\cup \{\pm z_1\}.$$
Let $\theta=\theta_\ell=\frac{\pi\ell^2}{2(\ell^3-1)}$, and let $R=R_\ell>1$ be such that 
\begin{equation}\label{eqn:Rell}
R^{2\ell}= R^2+R^{2/\ell}+2R^{1+1/\ell}\cos \frac{\pi(\ell+1)}{2(\ell^3-1)}.
\end{equation}

\begin{lemma}\label{lem:technical} For each odd integer $\ell\ge 3$,
$$2R_\ell \cos \frac{\theta_\ell}{\ell^2} > 2^{\frac{1}{(\ell-1)}}+ 2^{\frac{1}{\ell^2-\ell}}.$$
\end{lemma}
\begin{proof}
Put $\alpha= \frac{\pi (\ell+1)}{2(\ell^3-1)}$ and $\beta=\frac{\theta}{\ell^2}$. Using the assumption $\ell\ge 3$, it is easy to check that  $\alpha \le \frac{9\pi}{13}\frac{1}{\ell^2},$ and $\beta\le \frac{2}{\ell^3}$.
Therefore
$$(\cos \alpha)^{\ell^2} \ge \left(1-\frac{\alpha^2}{2}\right)^{\ell^2}\ge 1-\frac{\ell^2 \alpha^2}{2}>0.7,$$
where we used $(1-x)^{\ell^2}>1-\ell^2 x$ for $x\in (0,1)$. Consequently, $\cos \alpha>0.9$ and
\begin{equation}\label{eqn:cosalpha}
(1+\cos \alpha)(\cos \alpha)^{\ell^2}>1.
\end{equation}
By (\ref{eqn:Rell}),
$$R^{2\ell} =(R+R^{1/\ell} \cos \alpha)^2+ (R^{1/\ell}\sin \alpha)^2> (R+R^{1/\ell} \cos \alpha)^2,
$$
hence
$$R^{\ell}> R+R^{1/\ell} \cos \alpha,$$
which implies $R^{\ell}> 1+\cos \alpha$
since $R>1$. By (\ref{eqn:cosalpha}), we obtain
$$R^{\ell}> R+1.$$
It follows that $R^\ell>2$ and consequently $R^\ell > 1+2^{1/\ell}$. Therefore
\begin{equation}\label{eqn:Rcosbeta}
(R\cos \beta)^\ell > (1+2^{1/\ell})(1-\ell\beta^2/2)=(1+2^{1/\ell})(1-2/\ell^5)
\end{equation}

{\em Case 1.} $\ell=3$. By direct computation, we deduce from (\ref{eqn:Rcosbeta}) that $(R\cos \beta)^3> 2.24$, hence
$2R\cos \beta> 2.61$. But $2^{1/2}+2^{1/6}<2.54 <2.61$.

{\em Case 2.} $\ell\ge 5$. Using $2^{1/\ell}> 1+\frac{1}{2\ell}$, we deduce from (\ref{eqn:Rcosbeta}) that $(R\cos \beta)^\ell>2$. Thus it suffices to prove
$$ 2 \cdot 2^{1/\ell} > 2^{1/(\ell-1)}+ 2^{1/(\ell^2-\ell)}.$$
For this purpose, put $\delta=2^{1/(\ell^2-\ell)}-1\in (0, 0.5)$. Then
$$\frac{2^{1/(\ell-1)}+ 2^{1/(\ell^2-\ell)}}{ 2^{1/\ell}}=1+\delta+ \frac{1}{(1+\delta)^{\ell-2}}
< 1+\delta +\frac{1}{(1+\delta)^3}<2.$$
The proof is completed.
\end{proof}

 We say that a holomorphic motion $h_{\lambda}(x)$ of $P$ over $\D_r$ is {\em admissible} if the following hold for each $\lambda\in \D_r$:
\begin{enumerate}
\item [(A1)] $h_{\lambda}(-z_1)=-h_{\lambda}(z_1)$, $h_\lambda(0)=0$;
\item [(A2)] For each $x\in P$ with $0<x< w$, we have $|h_\lambda(x)|\le |h_{\lambda}(z_1)|$;
\item [(A3)] For each $x\in P$ with $x>w$, we have $h_{\lambda}(x)\in S_{\theta}$,
where $$S_\theta=\{re^{it}: r>0, |t|<\theta\};$$
\item [(A4)] For each $x\in P$ with $x<-w$, we have $h_\lambda(x)\in -S_{\theta}$;
\item [(A5)] $h_\lambda(c_1)\in -S_{\theta/\ell^2}$, $h_\lambda(c_2)\in S_{\theta/\ell}$;
\item [(A6)] $|h_\lambda(c_1)|> R$, $|h_\lambda(c_2)|>R^{1/\ell}$;
\item [(A7)] $|h_\lambda(x)|\le 2^{1/(\ell-1)}$ for all $x\in P$;
\item [(A8)] $|h_\lambda(z_1)|\le 2^{1/(\ell^2-\ell)}$.
\end{enumerate}

\medskip
{\bf Main Lemma.} {\em  Assume that $\ell\ge 3$ is an odd integer and let $f=f_{c_1}$ be as above.
Then any admissible holomorphic motions $h_\lambda$ of $P$ over $\D_r$ has a lift $\widehat{h}_\lambda$ which is
again an admissible holomorphic motion of $P$ over $\D_r$.
}
\medskip

\begin{proof}
{\bf Step 1.} It is clear that for each $x\in P$, $\lambda\mapsto \widehat{h}_\lambda(x)$ can be defined over $\D_r$ as a holomorphic map, so that
\begin{itemize}
\item $\widehat{h}_\lambda(0)=0,$
\item for $x>0$,
$(\widehat{h}_\lambda(x))^\ell= h_\lambda(f(x))-h_\lambda(c_1),$
\item for $x<0$,
$-(\widehat{h}_\lambda(x))^\ell=h_\lambda(f(x))-h_\lambda(c_1).$
\end{itemize}
Thus $\widehat{h}_\lambda$ satisfies (A1).

{\bf Step 2.} Let us prove that $\widehat{h}_\lambda$ is indeed a holomorphic motion of $P$ over $\D_r$. Arguing by contradiction, assume that there exists $x,y\in P$ and $\lambda_0\in \D_r$ such that $\widehat{h}_{\lambda_0}(x)=\widehat{h}_{\lambda_0}(y)$. Then using the assumption that $\ell$ is an odd integer, we must have $xy<0$ and
$$h_{\lambda_0}(f(x))-h_{\lambda_0}(c_1)=-(h_{\lambda_0}(f(y))-h_{\lambda_0}(c_1)).$$
Thus
$$\text{Re} h_{\lambda_0}(f(x))+ \text{Re} h_{\lambda_0}(f(y))= 2 \text{Re} h_{\lambda_0}(c_1).$$
Assume without loss of generality $x<0$ and $y>0$. Then
$f(x)\ge 0$. Since $h_\lambda$ satisfies (A2) and (A3), we have
$$\text{Re} h_{\lambda_0}(f(x)) \ge  -|h_{\lambda_0}(z_1)|.$$
On the other hand, since $h_\lambda$ satisfies (A7),
$$|\text{Re} h_{\lambda_0}(f(y))|\le |h_{\lambda_0}(f(y))|\le 2^{1/(\ell-1)}.$$
Therefore, we have
$$-2\text{Re} h_{\lambda_0}(c_1)\le |h_{\lambda_0}(z_1)|+2^{1/(\ell-1)}\le 2^{1/(\ell^2-\ell)}+2^{1/(\ell-1)},$$
where the last inequality follows from the property (A8) for $h_\lambda$.
However, this contradicts with Lemma~\ref{lem:technical} by the properties (A5) and (A6) for $h_\lambda$.


{\bf Step 3.} Let us prove that $\widehat{h}_\lambda$ satisfies the property (A2).
It suffices to show that for each $y=f(x)\in [c_1, -w]\cap P$,
$$|h_\lambda(y)-h_\lambda(c_1)|\le |h_\lambda(c_1)|.$$
Indeed, writing
$$\zeta=\frac{h_\lambda(y)}{h_\lambda(c_1)}=re^{it},$$
we have $r\le 2^{1/(\ell-1)}$ and $|t|<\theta (1+1/\ell^{2})$. Provided that $\ell\ge 3$, we have
$r<2\cos t$, which implies $|\zeta-1|<1$ and hence the desired estimate.

{\bf Step 4.} The property (A5) for $\widehat{h}_\lambda$ follows from
$(-\widehat{h}_\lambda(c_1))^\ell=h_\lambda(c_2)-h_\lambda(c_1)\in S_{\theta/\ell}$,
$\widehat{h}_\lambda (c_2)^\ell= h_\lambda(c_3)-h_\lambda(c_1)\in S_\theta$.

Similarly, for any $x\in P$ with $|x|>w$, since $|h_\lambda(c_1)|>|h_\lambda(z_1)|$ and
$$h_\lambda(f(x))\in S_{\theta} \cup \D_{|h_\lambda(z_1)|},$$
it follows that $h_\lambda(f(x))-h_\lambda(c_1)\in S_{\frac{\theta}{\ell^2}+\frac{\pi}{2}}$, and hence
$$\widehat{h}_\lambda(x)\in \pm S_{\frac{\theta}{\ell^3}+\frac{\pi}{2\ell}}=\pm S_\theta.$$ This proves that (A3) and (A4) hold for $\widehat{h}_\lambda$.

{\bf Step 5.} Let us prove the property (A6) for $\widehat{h}_\lambda$.
Indeed,
$$|\widehat{h}_\lambda(c_1)|^{2\ell}=|h_\lambda(c_2)-h_\lambda(c_1)|\ge R^2+1+2R\cos \frac{\pi(\ell+1)}{2(\ell^3-1)}=R^{2\ell},$$
since $|h_\lambda(c_1)|>1$, $|h_\lambda(c_2)|>1$ and $\angle h_\lambda(c_1) 0 h_\lambda(c_2) > \pi-\frac{\theta}{\ell}-\frac{\theta}{\ell^2}>\frac{\pi}{2}.$
This proves that $|\widehat{h}_\lambda(c_1)|>R$.
Using $\angle h_\lambda(c_1)0h_\lambda(c_3)> \pi/2$, we obtain
$|\widehat{h}_\lambda(c_2)|^\ell> |h_\lambda(c_1)|>R.$


{\bf Step 6.} We prove the properties (A7) and (A8) for $\widehat{h}_\lambda$.
Indeed, for any $x\in P$,
$$|\widehat{h}_\lambda(x)|^\ell=|h_\lambda(f(x))-h_\lambda(c_1)|\le 2\cdot 2^{1/(\ell-1)}=2^{\ell/(\ell-1)},$$
which implies that $|\widehat{h}_\lambda(x)|\le 2^{1/(\ell-1)}$. This proves (A7).
For $x=z_1$, we have $h_\lambda(f(x))=0$, and thus
$|\widehat{h}_\lambda(z_1) |^\ell \le 2^{1/(\ell-1)}.$
This proves (A8).

\end{proof}

\part*{Part D: The lifting property for some well-known families}
\section{Polynomials and rational functions}\label{sec:rationalmaps}
In this section we demonstrate that the method of Section~\ref{sec:lifting} on holomorphic maps also applies in the setting of rational maps on the
Riemann sphere. In this setting we can use the Measurable Riemann Theorem to prove the lifting property.
The main results obtained in this section can also be covered by other methods, see~\cite{Le, BE, LSvS}.

\subsection{Holomorphic perturbations}
Let $\textbf{P}_d$ and $\textbf{Rat}_d$ be collections of all monic centered polynomials and rational functions of degree $d\ge 2$; these sets are naturally parametrized by $\C^{d-1}$ and an open set in $P\C^{2d+1}$ respectively. In this and the next subsections $f$ is an arbitrary function either from $\textbf{P}_d$ or from $\textbf{Rat}_d$. In the latter case we assume without loss of generality that the orbits of critical points avoid the point at $\infty$.
Let $c_1, c_2, \cdots, c_\nu$ be all distinct (finite) critical points of $f$ with multiplicities $m_1, m_2, \cdots, m_\nu$ and let $v_j=f(c_j)$.
We define a holomorphic deformation $(f,f_{\textbf{w}},\textbf{p})_W$ of $f$ as follows.

If $f$ is a polynomial, there is a neighbourhood $W$ of $(v_1, v_2,\cdots, v_\nu)$ in $\C^\nu$ and a neighbourhood $W_f\subset \textbf{P}_d$ of $f$ such that for each $\textbf{w}\in W$, there is a unique polynomial $f_{\textbf{w}}\in W_f$, depending on $\textbf{w}$ holomorphically, and a holomorhic function $\textbf{p}=(p_1,p_2,\cdots,p_\nu):W\to \C^\nu$, such that

(i) $p_j(\textbf{w})$ is a critical point of $f_{\textbf{w}}$ of multiplicity $m_j$ and
$$\textbf{w}=(f_{\textbf{w}}(p_1(\textbf{w})),f_{\textbf{w}}(p_2(\textbf{w})),\cdots,f_{\textbf{w}}(p_\nu(\textbf{w}))),$$

(ii) $f_{(v_1,v_2,\cdots, v_\nu)}=f$, $p_j(v_1, v_2, \cdots, v_\nu)=c_j$.

For a proof, see \cite{Le0}, Proposition 1.

\medspace

Now let $f$ be a rational function.
We say that a rational map $g$ of degree $d$ is in the class $\textbf{Rat}_d^{\bf m}$ where $\textbf{m}=(m_1,\cdots,m_\nu)$
if $g$ has $\nu$ distinct critical points $c_1, c_2, \dots , c_\nu$ with multiplicities $m_1,m_2,\dots,m_\nu$ respectively.
\begin{theorem}\label{prop:cvrational}
$\textbf{Rat}^{\bf m}_d$ is a manifold of dimension $\nu+3$
and the functions defined by the critical values form a partial holomorphic coordinate system.
In other words, $\Psi\colon \textbf{Rat}_d^{\bf m} \ni g \mapsto (g(c_1),\dots,g(c_\nu))$ has rank $\nu$.
\end{theorem}
\begin{remark}
A direct elementary proof of
Theorem~\ref{prop:cvrational} is given in \cite{LSvS}. Here we derive it from \cite{Le0}.
\end{remark}
\begin{proof}
Let $\Lambda_{d, \nu}$ be a collection of functions $g\in \textbf{Rat}^{\bf m}_d$ with the following expansion $g(z)=\sigma_g z+m_g+O(1/z)$ as $z\to \infty$ for some $\sigma_g\not=0$ and $m_g\in \C$.
By Proposition 3 of \cite{Le0}, $\Lambda_{d,\nu}$ is a manifold of dimension $\nu+2$ and
$(\sigma_g,m_g,v_1(g),\cdots,v_\nu(g))$ is a holomorphic local coordinate of $g\in\Lambda_{d,\nu}$. Fix some $w\in \C\setminus P(f)$. Now, given $g\in \textbf{Rat}^{\bf m}_d$ close to $f$, the function $z\mapsto 1/\{g(1/z+w)-g(w)\}$ is in $\Lambda_{d,\nu}$. Therefore, the vector
$$\textbf{x}(g)=(g(w), D g(w), D^2g(w), v_1(g),\cdots,v_\nu(g))$$ defines a local holomorphic coordinate of $g\in \textbf{Rat}_d^{\textbf{m}}$.
\end{proof}
In the next corollary we define a holomorphic deformation $(f,f_{\textbf{w},Z}, \textbf{p}_Z)_W$ of the rational function $f$. It depends on a given set $Z=\{x_1, x_2, x_3\}\subset \CC$ of three distinct points so that $Df(x_i)\not=0$, $i=1,2,3$.
\begin{coro}\label{sect3p}
There is a neighbourhood $W$ of $(v_1, v_2,\cdots, v_\nu)$ in $\C^\nu$ and a neighbourhood $W_f\subset \textbf{Rat}_d$ of $f$ such that for each $\textbf{w}\in W$, there is a unique rational function $f_{\textbf{w}}=f_{\textbf{w},Z}\in W_f$, depending on $\textbf{w}$ holomorphically, and a holomorhic function $\textbf{p}=\textbf{p}_Z:W\to \C^\nu$, such that (i)-(ii) hold and also $f_{\textbf{w},Z}(x_i)=f(x_i)$ for $i=1,2,3$.
\end{coro}
\begin{proof} Let $S=\{g\in \textbf{Rat}_d^{\textbf{m}}: g(x_i)=f(x_i), i=1,2,3\}$. There exists a neighbourhood
$W$ of $f$ and a neighbourhood $U$ of the identity in the space of Moebius transformations,
so that for any $g\in W$ there exists a unique Moebius transformation $M_g\in U$ so that $M_g(x_i)=y_i$ where $y_i=g^{-1}\circ f(x_i)$ is the $g$-preimage of $f(x_i)$ close to $x_i$, $i=1,2,3$. Hence $g\circ M_g(x_i)=f(x_i)$ and therefore $g\circ M_g \in S$.
It follows that the map $\Phi\colon W\to S\times U$ defined by $g\mapsto (g\circ M_g, M_g)$
is a local diffeomorphism  with inverse $(g,M)\mapsto g\circ M^{-1}$. Since $U$ has dimension three,
$S\cap W$ is a codimension-three manifold of $\textbf{Rat}_d^{\textbf{m}}$.
Since $\Psi(g\circ M)=\Psi(g)$ for all $M\in U$, it follows from this and the previous theorem, that
$(\Psi|S)\colon S\to \R^\nu$ is a diffeomorphism.
%
\end{proof}
\subsection{Lifting holomorphic motions}
Let $(f,f_{\textbf{w}}, \textbf{p})_W$ be a holomorphic deformation which is defined in the previous Subsection. Here, if $f$ is a rational function,
then $f_{\textbf{w}}=f_{\textbf{w},Z}$ and $\textbf{p}=\textbf{p}_Z$ where $Z\subset \CC$ is a set of 3 distinct points such that $Df(x)\not=0$ for $x\in Z$.

In Section~\ref{sec:lifting}, under the assumption that $P(f)$ is a finite set, given a holomorphic motion $h_\lambda$ of $P(f)$, we defined a lift $\widehat{h}_\lambda$ with respect to the local holomorphic deformation $(f,f_{\textbf{w}},\textbf{p})_W$ which exists in a small disk around zero. We shall now show that the lift $\widehat{h}_\lambda$ exists globally (i.e. it exists as a holomorphic of $P(f)$ over $(\D, 0)$), even when $P(f)$ is an {\it infinite set}.
\begin{prop}\label{global} Let $h_\lambda$ be a holomorphic motion of $P(f)$ over $(\D, 0)$. If $f$ is a rational function, we assume additionally about the set $Z$ that $f(Z)=Z$, $Z\cap P(f)=\emptyset$ and $\infty\in Z$ and also $h_\lambda(z)\notin Z$ for all $z\in P(f)$.
Then there exists a holomorphic motion $\widehat{h}_\lambda(z)$ of $P(f)$ over $(\D, 0)$ such that for each $z\in P(f)$,
\begin{equation}\label{eqn:liftpoly}
f_{(h_\lambda(v_1), h_\lambda(v_2),\cdots, h_\lambda(v_\nu))}(\widehat{h}_\lambda(z))=h_\lambda(f(z))
\end{equation}
holds when $|\lambda|$ is small enough
and, moreover, if $f$ is rational, $\widehat{h}_\lambda(z)\notin Z$ for all $\lambda\in \mathbb  D$ and all $z\in P(f)$.
\end{prop}
\begin{proof}
By the $\lambda$-lemma we can extend $h_\lambda$ to a holomorphic motion over $(\D, 0)$ of the whole complex plane, subject to the following normalization: (a) if $f$ is a polynomial, for $|z|$ large enough, $h_\lambda(z)$ is holomorphic in $z$ and $h_\lambda(z)=z+o(1)$ near infinity and (b) if $f$ is a rational function, $h_\lambda(z)=z$ for all $z\in Z$.
(We can also require that the extended holomorphic motion $h_\lambda(z)$ is holomorphic near all superattracting periodic points for each $\lambda$.) Let $\mu_\lambda$ denote the complex dilatation of $h_\lambda:\C\to\C$ and let $\widetilde{\mu}_\lambda=f^*(\mu_\lambda)$. Since $\mu_\lambda$ depends on $\lambda$ holomorphically, so does $\widetilde{\mu}_\lambda$. Let $\widetilde{h}_\lambda$ denote the unique qc map with complex dilatation $\widetilde{\mu}_\lambda$ which satisfies $\widetilde{h}_\lambda(z)=z+o(1)$ near infinity for each $\lambda$ in case (a) and $\widetilde{h}_\lambda(z)=z$ for $z\in Z$ in case (b).
Note that (b) and the injectivity of $\widehat{h}_\lambda:\C\to \C$ implies that $\widehat{h}_\lambda(z)\notin Z$ for $z\in P(f)$.

{\bf Claim.} For $|\lambda|$ small enough, we have $$h_\lambda\circ f\circ \widetilde{h}_\lambda^{-1}=f_{(h_\lambda(v_1), h_{\lambda}(v_2),\cdots, h_\lambda(v_k))}.$$ Hence (\ref{eqn:liftpoly}) holds when $\lambda$ is sufficiently close to $\ast$ in $\Lambda$.

Indeed, for each $\lambda$,  the complex dilatation of $\widetilde{h}_\lambda$ is the  lift of that of $h_\lambda$,  and therefore the function
$g_\lambda:=h_\lambda\circ f\circ \widetilde{h}_\lambda^{-1}$ is holomorphic in $\CC$. It is a branched covering of degree $d$, so it is either a polynomial or a rational function of degree $d$. By the normalization of both $h_\lambda$ and $\widetilde{h}_\lambda$, $g_\lambda$ is either a monic centered polynomial or a rational function such that $g_\lambda(z)=f(z)$ for $z\in Z$. Clearly the critical values of $g_\lambda$ are $h_\lambda(v_i)$. The claim follows.
\end{proof}

%

\subsection{{\lq\lq}Positively oriented{\rq\rq} transversality}
Assume that $P=P(f)$ is finite so that $f$ is a marked map.

\begin{theorem}\label{thm:pora}  Let $\mathcal{A}$ be the transfer operator associated to the holomorphic deformation $(f, f_{\textbf{w}}, \textbf{p})_W$ of $f$. Here, if $f$ is rational, $f_{\textbf{w}}=f_{\textbf{w},Z}$, $\textbf{p}=\textbf{p}_Z$ where $Z\subset \CC$ is a set of $3$ distinct points such that $Df(x)\not=0$ for all $x\in Z$ and, additionally, $f(Z)=Z$, $Z\cap P=\emptyset$ (Such a set always exists.)
Then the spectral radius of $\mathcal{A}$ is at most $1$ and $1$ is not an eigenvalue of $\mathcal{A}$ unless $f$ is a flexible Latt\'es rational map.
Furthermore, if $f$ is a polynomial, there is $\xi\in (0,1)$ so that the spectral radius of $\mathcal{A}$ is at most $1-\xi$.

Assume additionally that $f$ has real coefficients and
all its critical points are real; moreover, if $f$ is rational, assume also $Z=\overline{Z}$. Then
the 'positively oriented' transversality property (\ref{eq:trans2}) holds:
\begin{equation}\label{transpora}
\frac{\det (D\mathcal{R}((v_1\cdots,v_\nu)))}{\prod_{j=1}^\nu Df^{q_j-1}(v_{j})}>0.
\end{equation}
\end{theorem}

\begin{remark}
The proof shows that $\xi=\xi(\#P,\delta)$ where $\delta$ is at most the minimal distance between pairs of different points of $P$, and
so in particular $\xi$  does not dependent on the degree of $f$.
\end{remark}
\begin{proof}
Assuming the first part of this theorem, the inequality~(\ref{transpora}) follows from Corollary~\ref{real}. Indeed, since $f$ is critically finite, each cycle of $f$ is either superattracting or repelling.
Since $\textbf{w}$ is a local coordinate for $f_{\textbf{w}}$ and $Z=\overline{Z}$  (if $f$ is rational) it follows that
the holomorphic deformation $(f,f_w,\textbf{p})_W$ is real and Corollary~\ref{real} applies.

Let us prove the first part of the theorem.
First, we consider the case of rational function and show that $(f,f_{\textbf{w}}, \textbf{p})$ has the lifting property . Let $h_\lambda^{(0)}$ be a holomorphic motion of $P$ over $(\D, 0)$.
Let us fix $\eps>0$ such that $h_\lambda(x)\notin Z$ for all $|\lambda|<\eps$ and all $x\in P$.
As $Z\cap P=\emptyset$, there is a Moebius $M$ so that $\tilde Z=M(Z)\ni\infty$ and $\tilde P=M(P)\subset \C$. Note that $\tilde Z\cap \tilde P=\emptyset$.
Consider the function $\tilde f=M\circ f\circ M^{-1}$ and its holomorphic deformation $(\tilde f, \tilde f_{\textbf{v}}, \tilde{\textbf{p}})$ which is defined by the set $\tilde Z$. Note that $\tilde f_{\textbf{v}}=M\circ f_{\textbf{w}}\circ M^{-1}$ and $\tilde{\textbf{p}}(\textbf{v})=M(\textbf{p}(\textbf{w}))$
where $\textbf{v}=M(\textbf{w})$ and $M(x_1,\cdots,x_\nu):=(M(x_1),\cdots,M(x_\nu))$.
Since $H_\lambda^{(0)}=M\circ h_\lambda^{(0)}$ is a holomorphic motion of $\tilde P$ over $(\Delta_\eps, 0)$ and $H_\lambda^{(0}(x)\notin \tilde Z$ for $x\in \tilde P$, Proposition~\ref{global} immediately implies that there is a sequence $(H_\lambda^{(k)})_{k=0}^\infty$ of holomorphic motions of $\tilde P$ over $(\Delta, 0)$ such that for $k=0,1,\cdots$, $H_\lambda^{(k+1)}$ is the lift of $H_\lambda^{(k)}$ for the triple $(\tilde f, \tilde f_{\textbf{v}}, \tilde{\textbf{p}})$ if $|\lambda|$ is small enough. It follows that if we define $h_\lambda^{(k)}=M^{-1}\circ H_\lambda^{(k)}$, then for all $k$,
$h_\lambda^{(k+1)}$ is the lift of $h_\lambda^{(k)}$ for the triple $(f,f_{\textbf{w}}, \textbf{p})$ if $|\lambda|$ is small enough, each function $h_\lambda^{(k)}(x)$ is meromorphic in $\lambda\in \Delta_\eps$ and omit
the values in $Z$.
As $\#Z=3$, the lifting property then follows from Montel's theorem.
Therefore, either (1) or (2) of the alternative of Theorem~\ref{thm:1eigen} holds.
Assume that (2) holds. We obtains a non-trivial local holomorphic family $f_{\text{w}(t)}$
of critically finite rational maps. It follows from the $\lambda$-lemma that $f_{\textbf{w}(t)}$ is conjugate to $f$ by a quasi-conformal homeomorphism $\phi_t$.
If $J_f$ has Lebesgue measure zero then $\phi_t$ is conformal almost everywhere, hence, is a Moebius transformations
which is close to the identity. At it must fix the points of $Z$, it is the identity map, a contradiction. If the measure of $J_f$ is positive, $h_t$ gives rise to an invariant line field on $J_f$
and since the postcritical set of $f$ is finite,  $f$ must be a flexible Latt\'es map
(see e.g. Corollary 3.18 of \cite{McM}). Thus the theorem is proved for rational maps.

Now, let $f$ be a (monic centered) polynomial. We need
\begin{lemma}\label{lem:robustpol}
There exists $R_\ast>0$ such that for any non-linear monic centered polynomial $g$, if all critical values of $g$ lie in $\overline{B}(0,R_\ast)$ then $g^{-1}(B(0,2R_\ast))\subset B(0,R_\ast)$.
\end{lemma}
\begin{proof} We begin with

{\bf Claim.}
There exist $R_\ast>1$ and $h_\ast>1$ as follows. Let $d\ge 2$ be an integer and $\psi: \{|w|>h\}\to \C$ a univalent map where $h\ge 1$ and $\psi(w)=w+O(1/w)$ as $w\to \infty$. If $h>1$, suppose there is $w_c$ such that $|w_c|=h^d$ and $|\psi(w_c)|\le R_\ast$.
Then $h<h_\ast$, moreover, $\C\setminus \psi(\{|w|>h\}\subset \{|z|<R_\ast\})$
and
for every $w$ if $|w|>h$ and $|\psi(w)|=R_\ast$ then $|\psi(w^d)|\ge 2R_\ast$.

The claim being rewritten for the function $h/\psi(h/z)$ follows then easily from the Koebe distortion theorem. Details are left to the interested reader.

Let $R_\ast$ be as in the Claim. Consider the map $\psi=B^{-1}$, where $B$ is the B\"ottcher function of $g$ which is defined in a neighbourhood of infinity and normalized such that $B(z)=z+O(1/z)$ as $z\to \infty$
(this is possible as the polynomial $g$ is monic and centered).
Then $\psi$ extends to a univalent map $\psi:\{|w|>h\}\to \C$ for some minimal $h\ge 1$ and $\psi(w)=w+O(1/w)$ as $w\to \infty$.
Note that if $h>1$ then there is $w_c$ such that $|w_c|=h^d$ and $\psi(w_c)$ is a critical value of $g$. Assume $|u|\le R_\ast$ for all critical values $u$ of $g$. Then the point $w_c$ is as in the Claim. Now
assume the lemma does not hold for $g$. Then there exists $z$ such that $|z|=R_\ast$ and $|f(z)|<2 R_\ast$. By the Claim, there is $|w|>h$ such that $|\psi(w)|=R_\ast$ and, hence, $|f(z)|=|\psi(w^d)|\ge 2 R_\ast$, a contradiction.
\end{proof}
Lemma~\ref{lem:robustpol} allows us to employ Lemma \ref{lem:perturb2spectrum} on perturbations to the triple $(f,f_{\textbf{w}},\textbf{p})$. Indeed, let $Q$ be a polynomial as in Lemma~\ref{lem:perturb2spectrum}. Let us choose $\xi\in (0,1)$ in such a way
that the map $\varphi_\xi(z)=z-\xi Q(z)$ is invertible on $B(0,R_\ast)$ and
$\varphi_\xi^{-1}(B(0,R_\ast))\subset B(0,2R_\ast)$. By Lemma~\ref{lem:perturb2spectrum}, the spectral radius of $\mathcal{A}$ is at most $1-\xi$ if we show that the triple
$(\varphi_\xi\circ g, \varphi_\xi\circ f_{\textbf{w}}, \textbf{p}\circ \psi_\xi)$ has the lifting property . To this end, let $h_\lambda^{(0)}$ be a holomorphic motion of $P$ over $(\D,0)$.
Since $P$ is finite, $P\subset B(0,R_\ast)$. Hence, there is $\eps>0$ so that $h_\lambda^{(0)}(x)\in B(0,R_\ast)$ for each $x\in P$ and $|\lambda|<\eps$.  We show by induction that
for each $k\ge 0$ the lift $h_\lambda^{(k+1)}$ of $h_\lambda^{(k)}$ for the triple is well-defined in $\Delta_\eps$ and $h_\lambda^{(k+1)}(x)\in B(0,R_\ast)$.
Assume $h_\lambda^{(k)}$ is well-defined in $\Delta_\eps$ and $h_\lambda^{(k)}(x)\in B(0,R_\ast)$. Consider a new holomorphic motion $\varphi_\xi^{-1}\circ h_\lambda^{(k)}$
of $P$ over $(\Delta_\eps, 0)$. Observe that it takes values in $B(0,2R_\ast)$. By Proposition~\ref{global}, there exists a holomorphic motion $h_\lambda^{(k+1)}$ of $P$ over $(\Delta_\eps, 0)$ such that for each $x\in P$,
$$
f_{(h_\lambda^{(k)}(v_1), h_\lambda^{(k)}(v_2),\cdots, h_\lambda^{(k)}(v_\nu))}(h_\lambda^{(k+1)}(x))=\varphi_\xi^{-1}\circ h_\lambda(f(x))
$$
holds when $|\lambda|$ is small enough. By the Uniqueness theorem, this equality holds for every $\lambda\in \Delta_\eps$ and by Lemma~\ref{lem:robustpol} and the indunction hypothesis
$h_\lambda^{(k+1)}(x)\in B(0,R_\ast)$ for $|\lambda|<\eps$. This completes the indunction and therefore the proof of Theorem~\ref{thm:pora}.
\end{proof}

\section{Piecewise linear multimodal maps}\label{subset:linear}
Given $\epsilon\in \{1,-1\}$, a positive integer $\nu$ and $\underline \kappa=(\kappa_1,\dots,\kappa_{\nu+1})\in \R^{\nu+1}$ with $\kappa_i>0$, let us
introduce the class $\mathcal{L}^{\epsilon}_{\nu,\underline \kappa}$ of $\nu$-modal piecewise linear continuous maps $g: [-1,1]\to \R$ as follows:
\begin{itemize}
\item there are $s=s(g)>0$ and $c_i=c_i(g)$, $0\le i\le \nu+1$ so that $-1=c_0<c_1<\cdots<c_{\nu}<c_{\nu+1}=1$ and for each $i\in \{1,\cdots,\nu+1\}$, $g_{[c_{i-1},c_i]}$ is a linear
(i.e. affine) map with slope $s_i=\epsilon_i \kappa_i s$ where here and later $\epsilon_i=(-1)^{i-1}\epsilon$.
\item $g(-1), g(1)\in \{-1,1\}$ (so $g(-1)=-\epsilon_1, g(1)=\epsilon_{\nu+1}$).
\end{itemize}
For example, $\mathcal{L}^1_{1,(1,1)}=\{f_t\}_{t>0}$ where $f_t(x)=-t|x|+(t-1)$ is the tent family.

Denote $\textbf{v}(g)=(v_1,\cdots,v_\nu)$ where $v_i=g(c_i)$, $1\le i\le \nu$, the vector of the extremal values of $g$. Here and below $v_0\equiv g(-1)=-\epsilon_1$, $v_{\nu+1}\equiv g(1)=\eps_{\nu+1}$.

Let us show that  given a vector
$$(v_1,\dots,v_\nu)\in \R^\nu\mbox{ with }\epsilon_i (v_i-v_{i-1})>0\,\, \forall i=1,\dots,\nu+1,$$
there exists a map $g\in \mathcal{L}_{\nu,\underline \kappa}^{\epsilon}$ for which $\textbf{v}(g)=(v_1,\dots,v_\nu)$.  To see this, take the piecewise linear map $g$ with values
$v_0,v_1,\dots,v_\nu,v_{\nu+1}$ are to be determined points $-1=c_0<c_1<\dots<c_\nu<c_{\nu+1}=1$ with slope $s_i=\epsilon_i \kappa_i s$ on
$(c_{i-1},c_i)$ where $s$ is also to be determined.
It follows that
\begin{equation}
c_i-c_{i-1}=\frac{(v_i-v_{i-1})}{s_i}=\frac{(v_i-v_{i-1})\epsilon_i }{\kappa_i s}>0\label{cj}\end{equation}
and that the piecewise linear map $g$ is given by
\begin{equation} g(x)=v_{i-1}+s_i (x-c_{i-1})\mbox{ for }x\in (c_{i-1},c_i).\label{g}\end{equation}
Note that  $g\in \mathcal{L}_{\nu,\underline \kappa}^{\epsilon}$ if and only if
$$2=\sum_{i=1}^{\nu+1} (c_i-c_{i-1}) = \sum_{i=1}^{\nu+1} \frac{(v_i-v_{i-1})\epsilon_i }{\kappa_i s}$$
that is if and only if
\begin{equation} s=\sum_{i=1}^{\nu+1} \frac{(v_i-v_{i-1})\epsilon_i }{2\kappa_i}.\label{slope}\end{equation}

%
%

Let us now assume $g\in \mathcal{L}^{\epsilon}_{\nu,\underline \kappa}$ such that every turning point $c_{0,i}=c_i(g)$, $1\le i\le \nu$ is either periodic or eventually periodic.
Let $U=\cup_{x\in P\setminus P_0} U_x$ where $U_x$ is a small complex neighbourhood of $x$ so that $U_x$, $U_y$ are disjoint whenever $x\not=y$.
Then $g: U\to \C$ is a marked map with $P_0=\{c_{0,i}\}_{i=1}^\nu$ and $P=\{g^k(x)|x\in P_0, k\ge 0\}$. Let $W$ be a small neighbourhood of $\textbf{c}_1=(g(c_{0,1}),\cdots,g(c_{0,\nu}))\in \C^\nu$. Given $x\in P\setminus P_0$ there is a single $i=i(x)\in \{1,\cdots,\nu+1\}$ such that
$x\in [c_{i-1}(g),c_{i}(g)]$ (here $x$ can be an end point iff $x=\pm 1$).
We define a local holomorphic deformation $(g,G,\textbf{p})_W$ of $g$ naturally as follows.
Given $w=(w_1,\cdots,w_\nu)\in W$, let $w_0\equiv -\epsilon_1p,w_{\nu+1}\equiv \eps_{\nu+1}$  and
$$S=\sum_{i=1}^{\nu+1} \frac{(w_i-w_{i-1})\epsilon_i }{2\kappa_i}$$
Then
\begin{enumerate}
\item for $1\le j\le \nu$,
$$p_j(w)=\sum_{i=1}^j \frac{(w_i-w_{i-1})\epsilon_i/\kappa_i }{S}$$
let also $p_0(w)\equiv -1$,
\item for any $x\in P\setminus P_0$, if $z\in U_x$ then
$$G_w(z)=w_{i-1}+\kappa_i \epsilon_i S (z-p_{i(x)-1}(w)). $$
\end{enumerate}
Observe that if all $w_i$ are real (and $W$ is small enough) then by~(\ref{cj})-(\ref{slope}), $G_w\in \mathcal{L}_{\nu,\underline \kappa}^{\epsilon}$, $\textbf{v}(G_w)=w$ and $c_j(G_w)=p_j(w)$.

\begin{theorem}\label{thm:linear}
Let  $g\in \mathcal{L}_{\nu,\underline \kappa}^{\epsilon}$ be so that each turning point of $g$ is periodic or eventually periodic. Assume that $g$ is ergodic with respect to the Lebesgue measure.
Then the `positively oriented' transversality property (\ref{eq:trans2}) holds
for the real holomorphic deformation $(g,G_w,\textbf{p})_W$ defined above.
\end{theorem}

Here as usual, we say that $g$ is ergodic with respect to the Lebesgue measure if for any $g$-invariant Borel sets $A, B\subset [-1,1]$ both with positive Lebesgue measure, we have that $A\cap B$ has positive Lebesgue measure.  Let $(I_i)_{i=1}^N$ consist of the components of $[-1,1]\setminus P$. Then {\em $g$ is ergodic if and only if for any $1\le i, j\le N$, the forward orbit of $I_i$ and $I_j$ under $g$ intersect.}  Indeed, the only if part is clear. For the if part, assume $g$ is not ergodic. Then there exists forward invariant Borel sets $A, B\subset [-1,1]$ such that both $A$ and $B$ have positive Lebesgue measure, but $A\cap B$ has Lebesgue measure zero. Take a Lebesgue density point $a$ of $A$ and let $i$ be such that the orbit of $a$ visits $I_i$ infinitely often. Then
%
using the fact
that $A$ is forward invariant, $g$ is linear and that $P$ is finite, it follows that $A$ contains $I_i$ up to a set of Lebesgue measure zero.
Similarly, $B$ contains one of the intervals $I_j$ up to a set of Lebesgue measure zero. Since $A,B$ are forward invariant, it follows that the forward orbits of $I_i$ and of $I_j$ intersects at most at a set of Lebesgue measure zero. By the Markov property of $g$, this implies that the forward orbit of $I_i$ and $I_j$ are disjoint.

For $\nu=1$, this theorem provides a new proof of transversality for the family of tent family $f_t$, $t>1$, see ~\cite{Tsu0} for the first proof.

\begin{remark}
The assumption that $g$ is ergodic is needed in this result. Indeed take a piecewise affine map $g\colon [-1,1]\to [-1,1]$ with
$g\in \mathcal{L}^1_{\nu,\underline \kappa}$  so that $g$ maps $[-1,0]$ into itself and $[0,1]$ into itself. One can take $g$ so that
$g$ is expanding on each branch, and so that its turning points are eventually periodic.
By conjugating $g$ with a map $h_t:[-1,1]\to [-1,1]$ which is affine both $[-1,0]$ and on $[0,1]$ and so that $h_t(\pm 1)=\pm 1$, $h_t(0)=t$,
one obtains a family of maps $g_t\in \mathcal{L}_{\nu,\underline \kappa}$. It follows that the transversality property
does not hold for $(g,G_w,\textbf{p})_W$.
\end{remark}

\begin{proof}[Proof of Theorem~\ref{thm:linear}]
Let $h_\lambda=h_\lambda^{(0)}$ be a real holomorphic motion of $g(P)$ over $(\D,0)$. Then for each $k=1,2,\ldots$ there exists a holomorphic motion $h_\lambda^{(k)}$ of $g(P)$ over $(\D_{\eps_k},0)$ for some $\eps_k>0$ such that $h_\lambda^{(k)}$ is the lift of $h_\lambda^{(k-1)}$ over $\D_{\eps_k}$. All these holomorphic motions are real. Choose $\eps_0>0$ such that $-1<h_\lambda(x)<h_\lambda(y)<1$ for all $\lambda\in (-\eps_0,\eps_0)$ and $x,y\in g(P)$, $-1<x<y<1$. Then clearly, for each $k$, $h_{\lambda}^{(k)}$ allows analytic continuation to a neighbourhood of $(-\eps_0,\eps_0)$ and moreover, for all $k\ge 0$ and all $-1<x<y<1$,
\begin{equation}\label{eqn:lorhom}
-1<h_\lambda^{(k)}(x)<h_\lambda^{(k)}(y)<1 \,\, \mbox{ for all }  \lambda\in (-\eps_0,\eps_0).
\end{equation}
\begin{lemma}\label{lorfrac}
For every $k=1,2,\cdots$, $h_\lambda^{(k)}(x)$ is of the form
\begin{equation}
h^{(k)}_\lambda(x)=\frac{a_k+\sum_{y\in g(P)} a_{y,k}(x) h_\lambda(y)}{b_k+\sum_{y\in g(P)} b_{y,k} h_\lambda(y)}
\label{B5}\end{equation}
where all coefficients are real.
\end{lemma}
\begin{proof}
We prove by induction on $k$.
For $k=0$ this holds trivially.
So assume it holds for some $k\ge 0$.
Then for $x=c_{0,j}\in P_0$,
\begin{multline*}
h_\lambda^{(k+1)}(c_{0,j})=p_j(h_\lambda^{(k)}(c_{1,1}),\cdots,h_\lambda^{(k)}(c_{1,\nu}))=\\
2\dfrac{\sum_{i=1}^j (w_i-w_{i-1})\epsilon_i/\kappa_i }{\sum_{i=1}^{\nu+1} (w_i-w_{i-1})\epsilon_i /\kappa_i} =
\\
2\dfrac{\sum_{i=1}^j (h_\lambda^{(k)}(c_{1,i})-h_\lambda^{(k)}(c_{1,i-1}))\epsilon_i/\kappa_i }{\sum_{i=1}^{\nu+1} (h_\lambda^{(k)}(c_{1,i})-h_\lambda^{(k)}(c_{1,i-1}))\epsilon_i/\kappa_i}.
\end{multline*}
where we have $h_\lambda^{(k)}(c_{1,0})\equiv -1$ and  $h_\lambda^{(k)}(c_{1,\nu+1})\equiv \epsilon_{\nu+1}$.
Now using the induction hypothesis and that the denominator of (\ref{B5})  does not depend on $x$ we obtain
$$  h_\lambda^{(k+1)}(c_{0,j}) =
2\dfrac{\sum_{i=1}^j \sum_{y\in g(P)}[(a_{y,k}(c_{1,i})-a_{y,k}(c_{1,i-1}))h_\lambda(y)]\epsilon_i/\kappa_i }
{\sum_{i=1}^{\nu+1} \sum_{y\in g(P)}[(a_{y,k}(c_{1,i})-a_{y,k}(c_{1,i-1}))h_\lambda(y)]\epsilon_i/\kappa_i  }.
$$
and so the induction statement holds.
The analogous calculation for $x\in P\setminus P_0$ completes the proof of the lemma.
This proves the induction statement; note that this calculation relies on the assumption that the slope of each branch is a fixed
multiple $\kappa_i$ of $s$.
\end{proof}
Together with (\ref{eqn:lorhom}), the lemma implies that there exists $\eps>0$ such that each $h_\lambda^{(k)}$ extends to a holomorphic map on $\D_\eps$ and moreover, they are uniformly bounded in $\D_\eps$. This proves that $(g, G,\textbf{p})_W$ has the lifting property.
It follows that one of the alternatives in the conclusion of Theorem~\ref{thm:1eigen} holds.
Let us show by contradiction that the 2nd alternative of this
theorem does not take place in the present setting, completing the proof of the theorem.
Indeed, otherwise there is a continuous one-parameter family of maps $g_t\in \mathcal{L}_{\nu,\underline \kappa}$, $t\in (-\delta,\delta)$ for some $\delta>0$ so that $g_0=g$, all $g_t$ are different and critically finite and so that the iteneraries of the turning points of $g_t$ are the same as $g$.  Let $s(t)$ be the parameter $s$ corresponding to $t$. So $s(0)=s$.
In order to show that such a family $g_t$ cannot exist,  write $P\cup \{-1,1\}=\{a_0,\dots,a_n\}$ so that $0=a_0<a_1<\dots<a_{n-1}<a_{n}=1$
and write $v_i=a_i-a_{i-1}$, $i=1,\dots,n$.
Since $g$ is critically finite,  the interval $(a_{i-1},a_i)$ is mapped onto a union of such intervals, and so the latter union is equal to
$s\kappa_j$ times the former. Let $v$ be the column vector $v=(v_1,\dots,v_n)$ and note that $\sum v_i=2$.
It follows that there exists matrix $A$  (for which each row consists of $0$'s and a block of some $\kappa_j$'s), so that
$v=sAv$. So $\det(I-sA)=0$. Since the iteneraries of the turning points of $g_t$ are the same as $g$, we have $\det(I-s(t) A)=0$.
%
It follows that $s(t)=s$ for all $t$ close to zero (since $g_t$ is close to $g$).
Notice that $g_t$ and $g$ are conjugate. Indeed, neither of these maps
has wandering intervals, nor do they have periodic attractors (because otherwise their critical points cannot
be critically finite). By assumption there cannot be an integer $k$ so that $g^k(x)=x$ holds on an interval, and so the same holds for $g_t$.
Since the itinerary of the critical orbits of $g$ and $g_t$ are the same  it therefore follows that
$g_t$ and $g$ are topologically conjugate for each $t$. Let $h_t$ be the conjugacy
and take $x<y$ in an interval of monotonicity of $g$. Since $g$ is critically finite, there exists $\delta>0$
independent of $x,y$, and integers $n(x), n(y)$ so that

(a) there exists $z\in (x,y)$ so that  both $g^{n(x)}|(x,z)$ and  $g^{n(y)}|(z,y)$ are monotone
and so that $|g^{n(x)}(x)-g^{n(x)}(z)|\ge \delta$ and $|g^{n(y)}(z)-g^{n(y)}(y)|)\ge \delta$;
here $z$ is so that either $g^{n(x)}(z)$ or  $g^{n(y)}(z)$ is equal to a turning point of $g$ or $n(x)=n(y)$ and then we take $z=(x+y)/2)$.

\noindent
By linearity of $g^{n(x)}$, $g_t^{n(x)}$ and since $g$, $g_t$ have the same corresponding slopes,
$$ \dfrac{|h_t(x)-h_t(z)|}{|x-z|}= \dfrac{|h_t(g^{n(x)}(x))-h_t(g^{n(x)}(z))|}{|g^{n(x)}(x)-g^{n(x)}(z)|}$$
and by property (a) we obtain lower and upper bounds for the right hand side of this expression.
So we obtain lower and upper bounds for $ \dfrac{|h_t(x)-h_t(z)|}{|x-z|}$ and similarly for $ \dfrac{|h_t(z)-h_t(y)|}{|z-y|}$,
and therefore for  $ \dfrac{|h_t(x)-h_t(y)|}{|x-y|}$.
It follows that $h_t$ is bi-Lipschitz, and in particular absolutely continuous. Hence $h_t$ is almost everywhere differentiable.
Since $h_t\circ g = g_t \circ h_t$, it follows that $Dh_t(g(x)\cdot Dg(x)=Dh_t(x) \cdot Dg_t(h_t(x))$  for Lebesgue a.e. $x$.
Since $g$ and $g_t$ have the same slope on corresponding branches,
it follows that $Dh_t(g(x)=Dh_t(x)$ and so $Dh_t$ is Lebesgue a.e. $g$-invariant.
Since $g$ is assumed to be ergodic with respect to the Lebesgue measure,  it follows that pthere exists $c\in \R$ so that
$Dh_t=c$ Lebesgue almost everywhere.  Hence by Newton-Leibnitz for absolutely continuous functions,
$h_t(x)=-1+\int_{-1}^x Dh_t(u) du=-1+c(x+1)$.  Since $h_t(1)=1$ it follows that $c=1$.
So $g_t=g$, a contradiction. Therefore, Corollary~\ref{real} applies as $g$ has no attracting or neutral periodic orbits.
\end{proof}

\section{Families of real maps with discontinuities}\label{sec:familiesdiscont}

In this Section we will give some examples which show that the methods developed
in this paper also apply to maps with discontinuities. So assume that
a real map $g$ is discontinuous at one point $c$, but that $c_1^+=\lim_{x\downarrow c}g(x)$
and $c_1^-=\lim_{x\uparrow c}g(x)$ are well-defined. If both these points are eventually periodic,
or eventually mapped onto $c$, then let $P$ be the corresponding forward orbits of $c^-$ and $c^+$.
To put this in the frame work of marked maps,  consider $c_{0,1}=c^-,c_{0,2}=c^+$ as two separate points
(which are identified) and take $P_0=\{c_{0,1},c_{0,2}\}$. Note that $g$ is not assumed to be holomorphic near $c$.
Now take a holomorphic motion of $P$ with $h_\lambda(c_{0,i})=c_{0,i}=c$ for all $\lambda$ and $i=1,2$.
We can still take the lift as in  (\ref{eq:lift2}) and also in
(\ref{eq:lift1}) provided we define  ${\bf p}(c_1^+(\lambda))\equiv {\bf p}(c_1^-(\lambda))$ correctly.

\subsection{Piecewise affine Lorenz maps}\label{subsec:affinelorenz}

Consider the family of affine Lorenz maps $f_{t,c}(x)=tx+(t-1)$ for $x\in [-1,c)$ and $f_{t,c}(x)=tx-(t-1)$ for $x\in (c, 1]$,
parametrised by $t\in (1,2]$ and $c\in (-1,1)$.
We say that $c^-$ and $c^+$ are both eventually periodic, if some iterate of $c_{1,1}=\lim_{x\uparrow c} f_{t,c}(x)$  is either
equal to $c$ or to a periodic orbit, and the same holds for $c_{1,2}=\lim_{x\downarrow c} f_{t,c}(x)$.

\begin{theorem}\label{thm:lorenz-twoparameter}
Assume that $c^-$ and $c^+$ are (both) periodic or eventually periodic. Then the spectral radius of the
corresponding transfer matrix $\mathcal{A}$ is  at most $1$, and $1$ is not an eigenvalue of $\mathcal{A}$;
in particular  the transversality condition (\ref{eq:trans2})
holds for the two parameter family $f_{t,c}$.
\end{theorem}

To clarify what we mean by (\ref{eq:trans2}) in the above theorem, write $c_{0,1}=c^-, c_{0,2}=c^+$,  and consider $f_{t,c}$ as a marked map $g$ for which $P_0=\{c^-,c^+\}$
and $P$ is the union of $P_0$ and the forward orbits of $c_{1,1}$ and $c_{1,2}$. So $P$ is a finite forward invariant set by
assumption. Take $W\subset \C^2$ be a small neighbourhood of $(c_{1,1},c_{1,2})$ and define the holomorphic maps $G_-(w,z)=tz+(t-1)$  and $G_+(w,z):=tz-(t-1)$,
where $w\in W$, $t=1+\frac{w_1-w_2}{2}$ and $\textbf{p}(w)=\frac{w_1+w_2}{2t}$.
For these choices of $t$ and ${\bf p}$ it follows that $G_-(w,{\bf p}(w))=w_1$ and $G_+(w,{\bf p}(w))=w_2$.
In particular, for $w=(c_{1,1},c_{1,2})$, $[-1,1]\ni z\mapsto G_\pm(w,z)$ agrees with $g$.
So we consider $G_\pm(w,x)$ as a local holomorphic deformation of $g$, and the map $\mathcal R$ in (\ref{eq:trans2}) is defined as in Subsection~\ref{subsec:transA}.
So for  example, if $c_1$ and $c_2$ have period $q_1$ and $q_2$, the above theorem gives that
$$\mathcal R(w_1,w_2)=(f_{t,c}^{q_1}(c_-)-c,f_{t,c}^{q_2}(c_+)-c),$$ where
$t=1+\frac{w_1-w_2}{2}$ and $c=\frac{w_1+w_2}{2t}$, is locally invertible for $w=(w_1,w_2)$ near $(c_1,c_2)$.

To prove this theorem, consider {\em real holomorphic motions}, i.e. so that  $h_\lambda(x)\in (-1,1)$  for $\lambda$ real.
Next let $h_\lambda^{(k+1)}$ be the lift of $h_\lambda^{(k)}$ and write $c_i^{(k)}(\lambda)=h^{(k)}_\lambda(c_i)$.
Here the lift $h^{(k)}_\lambda$ is defined as in equations (\ref{eq:lift1}) and (\ref{eq:lift2})
but in (\ref{eq:lift2}) we take  for $w_1>0>w_2$ (i.e. $t>1$),
$G_-(w,x)$ for $x<c$ or $G_+(w,x)$ for $x>c$ the linear functions from above.

 So arguing exactly
as in the previous section we obtain:

\begin{theorem}\label{lorlin}
The transfer operator $\mathcal{A}$ associated to the deformation $(g,G,\textbf{p})_W$ of the marked Lorenz map $g$ has spectral radius at most $1$ and $1$ is not an eigenvalue of $\mathcal{A}$.
\end{theorem}

Theorem~\ref{thm:lorenz-twoparameter} follows from this theorem.

\medskip
Note that monotonicity of entropy for the symmetric family of Lorenz maps  $f_{t,0}$
does not follow directly from this theorem. However, this follows immediately from the monotonicity for the tent family $T_t(x)=-t|x|+(t-1)$.
Indeed for any $x$, one has $f^n_t(x)=\pm T_t^n(x)$ and therefore either $f^n_t(x)=T_t^n(x)$ or $f^n_t(x)=T_t^n(-x)$.
Hence $h_{top}(f_t)=h_{top}(T_t)$,

\subsection{Lorenz maps with a flat critical point}

Fix real numbers $\ell\ge 1$ and $b> 2(e\ell)^{1/\ell}$ and consider the two-parameter family
$$f_{c_1,c_2}(x)=\left\{\begin{array}{ll}
-b e^{-1/|x|^\ell} +c_1 & \mbox{ if } x<0\\
b e^{-1/|x|^\ell} + c_2 & \mbox{ if } x>0.
\end{array}
\right.
$$
parameterised by $c_2<0<c_1$. Note that $f$ is increasing on each component of $\R\setminus \{0\}$.
Let $\beta\in (0, \ell^{1/\ell})$ be the number defined in Subsection~\ref{subsec:flatcritical}. Then $f_{\beta,-\beta}(\pm\beta)=\pm\beta$.

Fix $-\beta\le c_2<0<c_1\le \beta$ so that both $0^-$ and $0^+$ are periodic
or eventually periodic for $f_{c_1,c_2}$. Consider $g=f_{c_1,c_2}$ as
a marked map with $P_0=\{0^-,0^+\}$ and with $P$ the forward iterates of $P_0$.
For $w=(w_1,w_2)$ near $(c_1,c_2)$ define $G_- (w,x)$ to be equal to $-b e^{-1/|x|^\ell} +w_1$
and $G_+(w,x)=b e^{-1/|x|^\ell} +w_2$, define ${\bf p}\equiv 0$.

\begin{theorem}\label{lorflat}
The spectral radius of the operator $\mathcal{A}$ associated to the triple $(g,G,\textbf{p})_W$ is less than $1$.
The transversality condition (\ref{eq:trans2}) holds for the two-parameter family $f_{c_1,c_2}$.
\end{theorem}
\begin{proof}
Let $R>0$, $U=U^+\cup U^-$ and $V$  be as in Lemma~\ref{lem:flatext}
and define   $F_{-}(z):=-b e^{-1/z^\ell} +c_1$, $F_{+}(z):=b e^{-1/z^\ell} +c_2$
and $V_-=B(c_1,R)$ and $V_+=B(c_2,R)$.
Exactly as in Lemma~\ref{lem:flatext}, one has $\diam(U)<R$  and that
$F_{-,c_1}\colon  U_-\to V_-\setminus \{c_1\}$
and $F_{+,c_2}\colon U_+\to V_+\setminus \{c_2\}$
are unbranched covering maps. Since $\diam(U)<R$, for each
$(c_1,c_2)\in W:=U_+\times U_-$, one has that
$$U \subset \overline{S}_- \subset V_{-}\mbox{ and } U \subset \overline{S}_+ \subset V_+, $$
where $S_-=B(c_1,R)$ and $S_+=B(c_2,R)$.  The proof of the theorem is now a direct generalisation of the proof of Theorem~\ref{single}.
\end{proof}

Monotonicity of the topological entropy of the family of symmetric maps $f_{c,-c}$ follows as in the previous
subsection from the corresponding theorem for unimodal maps in Subsection~\ref{subsec:flatcritical}.


\bibliographystyle{plain}             

\end{document}